\documentclass{article}

\usepackage{xcolor} 
\usepackage{hyperref}

\hypersetup{
    colorlinks=true,        
    linkcolor=blue,         
    citecolor=purple,        
    urlcolor=magenta,       
}

\usepackage{amsthm}
\usepackage{amsmath,amssymb,amsfonts}
\usepackage{bm}
\usepackage{graphicx}
\usepackage{epstopdf}
\usepackage{tikz,pgf}
\usetikzlibrary{patterns}
\usepackage{pgfplotstable}
\usepackage{booktabs}
\usepackage{caption}
\usepackage{subcaption}

\usepackage{algorithmic}
\usepackage{listings}
\usepackage{xcolor}
\usepackage{multicol}
\usepackage[space,noadjust,nocompress]{cite}

\usepackage[
  left=25mm,
  right=25mm,
  top=30mm,
  bottom=30mm
]{geometry}

\numberwithin{equation}{section}
\numberwithin{figure}{section}
\numberwithin{table}{section}
\newtheorem{thm}{Theorem}[section]

\newtheorem{cor}[thm]{Corollary}
\newtheorem{lem}[thm]{Lemma}

\newtheorem{exa}[thm]{Example}

\newcommand{\sci}[2]{#1 \cdot 10^{#2}}

\title{
Computing exponential of tridiagonal Toeplitz matrices with applications to numerical solution of the heat equation
}

\author{
Mehdi Tatari\footnotemark[2]
\and
Majed Hamadi\footnotemark[3]
}

\newsavebox\verbbox
\date{}

\begin{document}

\maketitle

\renewcommand{\thefootnote}{\fnsymbol{footnote}}

\footnotetext[2]{Department of Mathematical Sciences, Isfahan University of Technology, Isfahan, Iran. ({\tt  mtatari@iut.ac.ir})}
\footnotetext[3]{Department of Mathematical Sciences, Sharif University of Technology, Tehran, Iran. ({\tt majed.hamadi77@sharif.edu})}

\begin{abstract}
The computation of the exponential of a tridiagonal matrix and its applications have always been of interest. One application considered here is when the method of lines is used to solve the heat equation, where the equation is transformed into a system of ordinary differential equations (ODEs), and this system has a solution that depends on the exponential of a tridiagonal Toeplitz matrix. Strang and MacNamara [SIAM Review, 56 (2014), pp. 525--546] presented an approximate method for computing the exponential of a symmetric tridiagonal Toeplitz matrix that appears in the solution of ODEs. Their method is based on approximating the entries of the exponential matrix with modified Bessel functions of the first kind at certain values, and accordingly, the exponential matrix is decomposed as the difference of a Toeplitz matrix and a Hankel matrix.  Here, we aim to extend this idea to the general case of tridiagonal Toeplitz matrices and stabilize the method by approximating the matrix exponential with a banded matrix, which makes the complexity of computing the exponential matrix independent of the matrix size. Additionally, we provide an error analysis for these methods and a bound for the entries of the exponential of the tridiagonal Toeplitz matrices. As a main contribution of this work, the idea is implemented to solve the heat equation, and the uniform stability of the method is proved. By using a splitting approach, the method is generalized for two-dimensional problems. Numerical illustrations demonstrate the efficiency of the new methods and bounds.
\end{abstract}



\section{Introduction}

Computation of the exponential of a matrix is an important problem because it has widespread applications in science and engineering \cite{BERGAMASCHI-CALIARI,Bottcher-Gutierrez,Crespo-Gutierrez,Bottcher-Crespo-Gutierrez, Lee}. Indeed, the matrix exponential is by far the most studied matrix function. Its study motivation stems from its key role in the solution of differential equations \cite{BERGAMASCHI-CALIARI,BERGAMASCHI-VIANELLO,N.Higham}. More precisely, the solution of a differential equation at time $t$ is usually expressed as $\exp(\mathbf{A}t)$ applied to a vector. The precise task and qualitative properties of the solution to a differential equation are important factors that affect the choice of method for computing the matrix exponential. There are some methods for computing the exponential of a matrix, such as the power series method, Jordan and Schur forms, and Cauchy integral approaches \cite{N.Higham}. Other well-known techniques for computing $\exp(\mathbf{A})$ include the scaling and squaring scheme, Padé approximation, interpolation method, and other methods using the Krylov subspace \cite{N.Higham,GOLUB-VAN,Moler,Kressner-Luce, Al-Mohy,Ben_Raz,Cortinovis}.  

Exponential of tridiagonal Toeplitz matrices play a crucial role in various scientific modelling \cite{Noschese}. These matrices appear in the discretization of the diffusion term and are ill-conditioned. In the numerical solution of differential equations, often a system with a tridiagonal Toeplitz matrix  needs to be solved. Although the system can be solved efficiently, the accuracy and qualitative properties of the resulting solutions remain challenging issues. Solutions based on the exponential matrix are an alternative approach that is considered, e.g., in \cite{Hairer}. In \cite{Hoseini}, a splitting method is presented for finding the exponential of tridiagonal matrices, which appear in the numerical solution of parabolic PDEs.

For the following heat equation
\[
u_t = au_{xx},
\]
we know mathematically (and by common experience, if $u$ represents, say, temperature) that $u(x,t)$ is bounded above and below by the extremes attained by the initial data and the values on the boundary up to time $t$. This fact is known as the maximum principle, and an engineering client for our computed results would be rather dismayed if they did not possess this property \cite{Morton}. In the numerical solution of differential equations, this property is achieved by stringent conditions on discretization parameters, especially for explicit methods. 

Here, based on the matrix exponential method presented in \cite{Strang}, we propose an explicit method for the heat equation using a banded approximation of this method. By using a splitting approach, we generalize the method to two-dimensional heat equation. We show that both methods preserve the maximum principle unconditionally. Also, we extend the study of the exponential of tridiagonal Toeplitz matrices to more general cases, including both symmetric and non-symmetric ones, and present error bounds for each. We show that these methods converge at a rate slightly faster than the geometric order and that the error bounds are sharp in numerical experiments. 

Furthermore, in order to improve the approximation methods and make them faster, we study bounds on the entries in the exponential of a tridiagonal matrix. Recently, many results have been established in this area, focusing on bounds for the entries of a function of a matrix, particularly for banded matrices \cite{Benzi-Golub, Benzi-Boito-Razouk, Ben_Boit, Ben_Raz, Iserles}. One of the results that closely related to the topic of this paper, provided by Iserles in \cite{Iserles}, gives a bound for the exponential of tridiagonal matrices. Additionally, in \cite{Benzi-Simoncini} the exponential of banded matrices was studied and some bounds for its entries were provided. Here, we provide a bound for the entries of the exponential of tridiagonal matrices that performs better in numerical results compared to the existing bounds. Using these studies, we approximate the exponential of tridiagonal matrices by banded matrices that use Bessel functions. We also use the Fast Fourier Transform (FFT) to compute more efficiently the multiplication of the approximation of the exponential of tridiagonal matrices by a vector, which is helpful in the method we provide for solving heat equations. 

The remainder of our work is organized as follows. We end this section with some needed notations.  In Section \ref{ESTTM}, we consider the problem of computing the exponential of tridiagonal Toeplitz matrices and provide an error analysis of this method. In Section \ref{sec:non-sym}, we present an approximation of the exponential of a non-symmetric tridiagonal matrix and analyze its error. In Section \ref{sec:decay}, we establish a bound for the entries of the matrix exponential and introduce a banded approximation method. In Section \ref{Method}, we present a numerical solution and analysis of the heat equation in one and two dimensions. 
In Section \ref{Nexamples}, we attest our study with numerical examples. Concluding remarks are presented in Section \ref{Conclusion}.

\subsubsection*{Notations} We define $\mathrm{Toeplitz}\{\mathbf{u}, \mathbf{v}\}$ as a Toeplitz matrix with $\mathbf{u}$ as its first column and $\mathbf{v}$ as its first row. If $\mathbf{u} = \mathbf{v}$, we simply use the notation $\mathrm{Toeplitz}\{\mathbf{u}\}$. We use $\mathrm{Hankel}\{\mathbf{v}\}$ for a Hankel matrix whose first column is $\mathbf{v}$ and whose last row is the \textit{reverted} $\mathbf{v}$. Finally, for integers $m$ and $n$ with $m < n$, we denote the set $\{ m, m + 1, \ldots, n-1,n \}$ by $\mathtt{m:n}$.

\section{Exponential of tridiagonal Toeplitz matrices}\label{ESTTM}
In this section, we first approximate the exponential of symmetric Toeplitz tridiagonal matrices. Let $\mathbf{T}_n \in \mathbb{C}^{n \times n}$ be a tridiagonal matrix defined as follows:
\begin{equation}\label{matrix_tridiag}
\mathbf{T}_n = \mathrm{tirdiag}(a,b,a) = 
\begin{bmatrix}
b& a &    \\
a & b& \ddots &  \\
 &\ddots   &\ddots&a \\
 &  & a & b \\
\end{bmatrix},
\end{equation}
where all sub-diagonal and super-diagonal elements in $\mathbf{T}_n$ are $a$, the main diagonal elements are $b$, and all other elements in $\mathbf{T}_n$ are zero.

Here, by using the eigenvalues of $\mathbf{T}_n$ and its eigenvectors, we aim to approximate the entries of $\exp(\mathbf{T}_n )$. Now, the eigenvalues of $\mathbf{T}_n$ are computed as follows:
\begin{align*}
 \lambda_k = b+ 2 a \cos\left(\dfrac{k\pi}{n+1} \right), \qquad k = 1 , \ldots , n,
\end{align*}
and an eigenvector corresponding to $ \lambda_k $ is given by $ v^{(k)} = \big(v_1^{(k)}, \ldots, v_n^{(k)}\big)^{T}$, where
\begin{align*}
v_j^{(k)} =\sin\left(\dfrac{j \pi k}{n+1}\right), \qquad j = 1 , \ldots , n.
\end{align*}
Then we can compute \(a_{ij}^n\) as the \(ij\)-th element of \(\exp(\mathbf{T}_n)\) as follows:
\begin{align*}
 a_{ij}^n &= \dfrac{2 \exp(b)}{n+1}\sum_{k=1}^n \exp\left(2a\cos\left(\dfrac{k\pi}{n+1}\right)\right) \sin\left(\dfrac{k \pi i}{n+1}\right) \sin\left(\dfrac{k \pi j}{n+1}\right)\nonumber\\
& =  \dfrac{\exp(b)}{n+1}\sum_{k=1}^n \exp\left(2a\cos\left(\dfrac{k\pi}{n+1}\right)\right) \cos\left(\dfrac{k \pi (i-j)}{n+1}\right) \\
&- \dfrac{\exp(b)}{n+1}\sum_{k=1}^n \exp\left(2a\cos\left(\dfrac{k\pi}{n+1}\right)\right) \cos\left(\dfrac{k \pi (i+j)}{n+1}\right)\nonumber.
\end{align*}
Then, we can easily see that for large values of \( n \), we can approximate \( a_{ij}^n \) as follows:
\begin{equation}\label{app}
\lim_{n \rightarrow +\infty} a_{ij}^n = e^b\big( I_{i-j}(2a) - I_{i+j}(2a)\big),
\end{equation}
where \( I_k(\cdot) \) are the modified Bessel functions of the first kind for \( k \in \mathbb{Z} \), which are defined as follows:
\[
I_k(x) = \sum_{s=0}^{\infty} \dfrac{1}{s! (k+s)!}\left(\dfrac{x}{2}\right)^{k+2s},
\]
and we can represent it in integral form as indicated in \cite{Arfken}:
\[
I_k(x) = \dfrac{1}{\pi}\int_0^\pi e^{x\cos(\theta)}\cos(k\theta) \mathrm{d}\theta.
\]

According to \eqref{app}, we can approximate $\exp(\mathbf{T}_n)$ entirely, as shown below:
\begin{align}\label{main}
\exp(\mathbf{T}_n)\approx\text{Toeplitz}\lbrace\mathbf{u}_n\rbrace-\text{Hankel}\lbrace\mathbf{v}_n\rbrace,
\end{align}
where $\mathbf{u}_n$ and $\mathbf{v}_n$ are vectors of length $n$, defined as follows:
\[\mathbf{u}_n=\big(e^bI_0(2a),\ldots,e^b I_{n-1}(2a)\big),~~~~~\mathbf{v}_n=\big(e^bI_{2}(2a),\ldots,e^bI_{n+1}(2a)\big).\]

Finally, in the rest of the discussion, we use the notation $\Psi(\mathbf{T}_n) \in \mathbb{C}^{n \times n}$ as defined in the following:
\begin{align}\label{approx_matrix}
\Psi(\mathbf{T}_n) = \text{Toeplitz}\lbrace\mathbf{u}_n\rbrace-\text{Hankel}\lbrace\mathbf{v}_n\rbrace.
\end{align}

In the rest of the paper, the method will be analyzed and extended for more complete cases that appear in the numerical solution of differential equations.

\subsection{Error analysis}\label{error:analysis:new}
In this section, we aim to establish an error bound for the approximation presented in the previous section, as given in \eqref{main}. Let \( v \) be a real or complex \( 2\pi \)-periodic function on the real line, and define
\[
I := \int_0^{2\pi}v(\theta)\mathrm{d}\theta.
\]
For any positive integer \( N \), the trapezoidal rule for approximating the integral takes the following form:
\[
T_N=\frac{2\pi}{N}\sum_{k=1}^Nv\left(\frac{2\pi k}{N}\right).
\]

The following theorem, proved in \cite[Page 211]{Kress-book}, provides a sharp error bound for the trapezoidal rule in approximating the integral of $2\pi$-periodic functions. This theorem is also discussed in \cite{Trefethen}.

\begin{thm}[\cite{Kress-book}]\label{trap}
Suppose $v$ is $2\pi-$periodic and analytic and satisfies $|v(\theta)|\leq M$ in the strip $ -s <Im (\theta)<s$ for some $s>0$. Then for any $N \geq 1$,
\[
|I-T_N|\leq\frac{4\pi M}{\exp(sN)-1},\\
\]
and the constant $4\pi$ is as small as possible.
\end{thm}
Now, suppose we want to approximate the following integral:
\begin{equation}\label{integral:2}
I_{ij}:=\frac{1}{2\pi}\int_0^{2\pi}\exp(b+2a\cos(\theta))\sin(i\theta)\sin(j\theta)\mathrm{d}\theta,
\end{equation}
with the following trapezoidal rule:
\[
T^{(i,j)}_{2n+2}:=\dfrac{1}{2n+2}\sum_{k=1}^{2n+2} \exp\left(b+2a\cos\left(\dfrac{2 k\pi}{2n+2}\right)\right) \sin\left(\dfrac{2 k \pi i}{2n+2}\right) \sin\left(\dfrac{2 k \pi j}{2n+2}\right).
\]

In integral \eqref{integral:2} the integrand function, $v(\theta)$, is $(1/2\pi)\exp(b+2a\cos(\theta))\sin(i\theta)\sin(j\theta)$. In this case, to apply Theorem \ref{trap}, the maximum value of $v(\theta)$ in the strip $-s \leq \operatorname{Im}(\theta) \leq s$ occurs at $\theta = \pm \mathbf{i}s$, where an upper bound for the integrand is given by:
\[
\vert v(\theta) \vert \leq (1/2\pi)\exp(\operatorname{Re}(b)+2|a|\cosh(s))\cosh(is)\cosh(js),
\]
and by using Theorem \ref{trap}, we obtain:
\begin{equation}\label{ineq:nn}
\left|I_{ij} -T^{(i,j)}_{2n+2} \right| \leq \frac{2\exp\left(\operatorname{Re}(b)+2|a|\cosh(s)\right)\cosh(is)\cosh(js)}{\exp(s(2n+2))-1},
\end{equation}
where this holds for all values of $s$. 
Since the following equality holds:
\[
\left| \frac{2}{\pi}\int_0^{\pi}\exp(b+2a\cos(\theta))\sin(i\theta)\sin(j\theta)\mathrm{d}\theta-a^n_{i,j}  \right|  = 2\left|I_{ij} -T^{(i,j)}_{2n+2} \right|,
\]
thus, from \eqref{ineq:nn}, it follows that:
\begin{equation}\label{est}
\begin{aligned}
\left | \frac{2}{\pi}\int_0^{\pi}\exp(b+2a\cos(\theta))\sin(i\theta)\sin(j\theta)\mathrm{d}\theta-a^n_{i,j}  \right |\leq \frac{4e^{\left(\operatorname{Re}(b)+2|a|\cosh(s)\right)}\cosh(is)\cosh(js)}{\exp(s(2n+2))-1}.
\end{aligned}
\end{equation}

The approximation $\Psi(\mathbf{T}_n)$ allows us to replace the elements of the exponential matrix with only $ n+2$ modified Bessel functions, which can be computed using a standard command in software or the trapezoidal rule with $ m \ll n$ quadrature points, independent of $n$.

We note that the matrix $\mathbf{T}_n$ is a per-symmetric matrix, which implies that both $ \exp(\mathbf{T}_n) $ and $\Psi(\mathbf{T}_n)$ are also per-symmetric. Thus, the entries satisfying $i + j > n + 1$ can be mapped to corresponding entries $(i', j')$ such that $i' + j' \leq n + 1$, having the same value in 
\begin{equation}\label{Error:matrix}
\mathbf{Er}(\mathbf{T}_n) := \exp(\mathbf{T}_n) - \Psi(\mathbf{T}_n).
\end{equation}
This fact allows us to apply the bound obtained for the $(i',j')$ entry in \eqref{est} to the $(i,j)$ entry as well, since they have the same values in $\mathbf{Er}(\mathbf{T}_n)$.

Let $\Vert \cdot \Vert_{\infty}$ denote the maximum absolute row sum of a matrix. Now, we present the following theorem.
\begin{thm}\label{approx2}
Let $ \mathbf{T}_n = \mathrm{tridiag}(a,b,a) \in \mathbb{C}^{n \times n} $, and let $ \mathbf{Er}(\mathbf{T}_n) $ be the matrix defined in \eqref{Error:matrix}. Then, for any $ n\geq 1 $, 
\begin{equation}\label{bound:err:estim:Gene}
\|\mathbf{Er}(\mathbf{T}_n)\|_{\infty}\leq \dfrac{4e^{\left(\operatorname{Re}(b)+2|a|\cosh(s)\right)}}{e^{s(2n+2)}-1} \Big(G_n(s)+G_n(-s)\Big).
\end{equation}
where $G_n(s) = \frac{e^{s(n+2)}-e^{2s}}{e^{s}-1}$, for $s>0$.
\begin{proof}
First, in \eqref{est}, we have $2\cosh(is)\cosh(js) = \cosh((i+j)s) + \cosh((i-j)s)$. Using the fact that the matrix $\mathbf{Er}(\mathbf{T}_n)$ is both symmetric and per-symmetric, the infinity norm of this matrix is therefore less than ``twice'' the sum of the error bounds in \eqref{est} for $i + j$ and $j - i$ ranging from 2 to $n + 1$. Thus, we have: 
\begin{align*}\label{lem:ineq:err}
\|\mathbf{Er}(\mathbf{T}_n)\|_{\infty}&\leq \dfrac{8e^{\left(\operatorname{Re}(b)+2|a|\cosh(s)\right)}}{e^{s(2n+2)}-1} \sum_{k=2}^{n+1}\cosh(ks)\\
&= \dfrac{4e^{\left(\operatorname{Re}(b)+2|a|\cosh(s)\right)}}{e^{s(2n+2)}-1} \left(G_n(s)+G_n(-s)\right),
\end{align*}
and the proof is complete. 
\end{proof}
\end{thm}

Theorem \ref{approx2} introduces a family of error bounds for the infinity norm of matrix $\mathbf{Er}(\mathbf{T}_n)$, depending on the parameter $s$. One suitable choice for $s$ in the bound \eqref{bound:err:estim:Gene} is the value $s = \ln((n+1)/|a|)$, where $(n+1) > |a|$, which performs sharply in numerical examples. However, in the next corollary, by using this value of $s$, we aim to simplify the bound from the previous theorem and make it more efficient.

\begin{cor}\label{approx}
Let $ \mathbf{T}_n = \mathrm{tridiag}(a,b,a) \in \mathbb{C}^{n \times n} $ be as defined in \eqref{matrix_tridiag}, and let $ \mathbf{Er}(\mathbf{T}_n) $ be the matrix defined in \eqref{Error:matrix}. Then, for large $ n $, 
\begin{equation}\label{bound:err:estim}
\|\mathbf{Er}(\mathbf{T}_n)\|_{\infty}\lesssim 2e^{\operatorname{Re}(b)}\left(\dfrac{|a|e}{n+1}\right)^{n+1}.
\end{equation}
\begin{proof}
Let $s = \ln((n+1)/|a|)$, where $n+1 > |a|$. Thus, by substituting the chosen value of $s$ into the hyperbolic cosine in \eqref{est}, we obtain $\cosh(ps) \sim 0.5 \big((n+1)/|a|\big)^p$ for any $p > 0$ as $n \to \infty$. Therefore,
{\small
\[
\dfrac{4e^{\operatorname{Re}(b)+2|a|\cosh(s)}\cosh(is)\cosh(js)}{e^{s(2n+2)}-1} \sim e^{\operatorname{Re}(b)}\left(\frac{|a|^2e}{(n+1)^2}\right)^{n+1}\left(\dfrac{n+1}{|a|}\right)^{i+j},~~~~~~~~n\rightarrow \infty.
\]}
Now, similar to the proof of Theorem \ref{approx2}, and using the fact that the matrix $\mathbf{Er}(\mathbf{T}_n)$ is per-symmetric, we have: \begin{align*}
\|\mathbf{Er}(\mathbf{T}_n)\|_{\infty}&\lesssim 2e^{\operatorname{Re}(b)}\left(\frac{|a|^2e}{(n+1)^2}\right)^{n+1}\sum_{k = 2}^{n+1}\left(\dfrac{n+1}{|a|}\right)^{k}\\
&\sim 2e^{\operatorname{Re}(b)}\left(\frac{|a|^2e}{(n+1)^2}\right)^{n+1}\left(\dfrac{n+1}{|a|}\right)^{n+1},~~~~~~~~~~~n\rightarrow \infty,
\end{align*}
and the proof is complete. 
\end{proof}
\end{cor}

According to the Corollary \ref{approx}, it is clear that $a$ has a greater effect on the accuracy compared to $b$. Also, since $(n+1)^{n+1}$ appears in the denominator of the error bound, this causes the approximation to converge very quickly as $ n$ increases.

Let us show the effectiveness of the bound from Corollary \ref{approx} with a numerical example.
\begin{exa}\label{ex1:matrix}
Let us consider $\mathbf{T}_n = \mathrm{tridiag}(1,-2,1)$, a type of matrix that appears in the numerical solution of the heat equation. We compare the error bound obtained in Corollary \ref{approx} with the ``infinity norm'' of the approximation error, $\|\mathbf{Er}(\mathbf{T}_n)\|_{\infty}$, in Table \ref{table:comp}. The results in the table show how effectively the bound in \eqref{bound:err:estim} estimates $\|\mathbf{Er}(\mathbf{T}_n)\|_{\infty}$. For example, for $n = 10$, this estimate is  $\sci{5.68}{-8}$, which is not far from the actual error $\sci{4.34}{-9}$ .

\begin{table}
\caption{Infinity norm of $\mathbf{Er}(\mathbf{T}_n)$, where $\mathbf{T}_n = \mathrm{tridiag}(1,-2,1)$, and the error bound from Corollary \ref{approx}.}
\label{table:comp}
\centering
\pgfplotstabletypeset[
  every head row/.style={
    before row=\toprule, 
    after row=\midrule,
  },
  every last row/.style={after row=\bottomrule},
  highlightcell/.style={
    postproc cell content/.append code={
      \ifnum\pgfplotstablerow=#1\relax
      \pgfkeysalso{@cell content/.add={$\bf}{$}}%
      \fi
    }
  },
  columns={Name, 3,5}, 
  columns/Name/.style={column name=$n$, 
   column type=c|},
  columns/3/.style={column name=$\|\mathbf{Er}(\mathbf{T}_n)\|_{\infty}$, column type=c|},
  columns/5/.style={column name=Error bound,  column type=c},
]{
Name 3 5
     1      0.079934            0.5
     2      0.033868        0.20135
     3     0.0094558       0.057727
     4     0.0017886       0.012855
     5    0.00028477      0.0023405
     6    3.8759e-05     0.00036043
     7    4.6617e-06     4.8092e-05
     8    5.0165e-07     5.6612e-06
     9    4.8858e-08     5.9619e-07
    10    4.3445e-09     5.6802e-08
}
\end{table}
\end{exa}

Here, we conclude this section, and in the next section, we aim to find an approximation for non-symmetric Toeplitz tridiagonal matrices.

\section{Non-symmetric Toeplitz tridiagonal matrices}\label{sec:non-sym}
In this section, we consider the case of matrices of the form \( \mathbf{A}_n = \mathrm{tridiag}(a,b,c) \), which denotes a tridiagonal Toeplitz matrix with \( a \) on the subdiagonal, \( b \) on the main diagonal, and \( c \) on the super-diagonal, while all other entries are zero. 

The technique discussed in the previous section can be extended to $\mathbf{A}_n $ through a simple similarity transformation, as stated in the following theorem.

\begin{thm}\label{jotri}
For nonzero $a$ and $c$, the tridiagonal matrix $\mathbf{A}_n = \mathrm{tridiag}(a,b,c)$ satisfies the following similarity relation:
\[
\mathbf{A}_n = \mathbf{D}_n \mathbf{T}_n \mathbf{D}_n^{-1},
\]
where $\mathbf{D}_n=\mathrm{diag}\big(\rho_1,\ldots,\rho_n\big)$, $\rho_k =\left( \sqrt{\dfrac{a}{c}}\right)^{k-1}$, and $\mathbf{T}_n=\mathrm{tridiag}\left(c\sqrt{\dfrac{a}{c}}, b, c\sqrt{\dfrac{a}{c}}\right)$.
\end{thm}
\begin{proof}
The proof is easily obtained using matrix multiplication.
\end{proof}
Let $\mathbf{A}$ and $\mathbf{B}$ be two $n \times n$ matrices. The Hadamard product of $\mathbf{A}$
and $\mathbf{B}$  is defined by $ (\mathbf{A}\circ \mathbf{B})_{ij} = (\mathbf{A})_{ij} (\mathbf{B})_{ij}$ for all $i,j\in\{1,\ldots,n\}$. In the following theorem, we use the Hadamard product to find an approximation for the elements of $\exp(\mathbf{A}_n)$ by utilizing the results from the previous section on the symmetric case.

\begin{thm}\label{Hadam}
Let $ \mathbf{A}_n = \mathrm{tridiag}(a, b, c) \in \mathbb{C}^{n \times n} $. The matrix $ \exp(\mathbf{A}_n) $ can be approximated by $\mathbf{F}_n \circ \Psi\left(\mathbf{T}_n \right)$ with the following error bound:
\[
\Vert \exp(\mathbf{A}_n) -  \mathbf{F}_n \circ \Psi\left(\mathbf{T}_n \right) \Vert_{\infty} \leq  2\Delta e^{\operatorname{Re}(b)}\left(\dfrac{|z|e}{n+1}\right)^{n+1},
\]
where $ \mathbf{T}_n = \mathrm{tridiag}(z, b, z) $, with $ z = c \sqrt{\frac{a}{c}} $, and $ \Psi(\mathbf{T}_n) $ as defined in \eqref{main}, and also $ \mathbf{F}_n = \mathrm{Toeplitz} \left\lbrace \mathbf{u}, \mathbf{v} \right\rbrace $, where $ \mathbf{u} = (a_0, a_1, \ldots, a_{n-1}) $, $ \mathbf{v} = (a_0, a_{-1}, \ldots, a_{-(n-1)}) $, with $ a_{k} = \left( \sqrt{\frac{a}{c}} \right)^{k} $, and $ \Delta = \max_{i,j} |(\mathbf{F}_n)_{ij}| $. 
\end{thm}
\begin{proof}
According to \eqref{main} and Theorem \ref{jotri}, we have
\[
\exp(\mathbf{A}_n) \simeq \mathbf{D}_n \Psi\big(\mathbf{T}_n \big)\mathbf{D}_n^{-1},
\]
also, the $ij$-th element of the right-hand side of the last approximation can be rewritten as follows:
\[
\dfrac{\rho_i }{\rho_j} (\Psi\big(\mathbf{T}_n))_{ij} = \left(\sqrt{\dfrac{a}{c}}\right)^{i-j}(\Psi\big(\mathbf{T}_n))_{ij}.
\]
Therefore, we can easily obtain that $\exp(\mathbf{A}_n) \simeq \mathbf{F}_n \circ \Psi(\mathbf{T}_n)$. For the error bound, we can use Corollary \ref{approx} and the fact that for matrices $C$ and $D$, we have $ \Vert C \circ D \Vert_{\infty} \leq \max_{i,j} |(C)_{ij}| \Vert D \Vert_{\infty}$. Then, the proof is complete.
\end{proof}
Although the previous theorem provides a method for computing the exponential of nonsymmetric Toeplitz matrices, there are challenges in its practical implementation. Regardless of the values of $a$, $b$, and $c$, some entries of $\mathbf{F}_n$ exceed $1$, and in computations for large values of $n$ may lead to numerical issues such as multiplying very large numbers by very small ones in $\mathbf{F}_n \circ \Psi(\mathbf{T}_n)$. On the other hand, computing each element of the exponential matrix is not necessary because they are significantly small. There are well-known results in this area, called the \textit{decay bound of matrix function elements} \cite{Ben_Boit, Benzi-Boito-Razouk, Benzi-Golub, Ben_Raz}. These results find uniform exponential decay bounds for elements of matrix functions. For our case of $ \exp(\mathbf{A}_n)$, there exist constants $C > 0$ and $0 < q < 1$, such that each $ij$-th element of this matrix is bounded by $Cq^{|j-i|}$. In the next section, inspired by the decay bound results, we present a new strategy for addressing these issues and approximating $\exp(\mathbf{A}_n)$ by a banded matrix.
\section{Decay estimates for exponential of tridiagonal matrices}\label{sec:decay}
Here, we show that the exponential of a tridiagonal matrix can be approximated by a banded matrix, which is leading to stabilization and a reduction in the computational complexity of the method. To achieve this, we first establish a bound for the entries of $\exp(\mathbf{T}_n)$ and then come back to the main problem. 
\subsection{Decay results}
Based on the results obtained in the previous sections, we can derive a bound for the entries of the exponential of tridiagonal matrices, which is presented in the following theorem.
\begin{thm}\label{thm:ineq_1first2}
For symmetric matrix $\mathbf{T}_n = \mathrm{tridiag}(a,b,a)$ and $i,j\in \{1,\ldots,n\}$:
\begin{align*}
\left\vert \big(\exp(\mathbf{T}_n)\big)_{ij} \right\vert &\leq 2e^{\operatorname{Re}(b)}\left(\dfrac{|a|e}{n+1}\right)^{n+1}+ e^{\operatorname{Re}(b)} \left(I_{|i-j|}(2|a|) +  I_{i+j}(2|a|)\right).
\end{align*}
\begin{proof}
The proof can be easily obtained using the following inequality:
\[
\left\vert \big(\exp(\mathbf{T}_n)\big)_{ij} \right\vert \leq  \Vert \exp(\mathbf{T}_n) - \Psi(\mathbf{T}_n) \Vert_{\infty}+\vert (\Psi(\mathbf{T}_n))_{ij} \vert,
\] 
and by applying Corollary \ref{approx} and the definition of $\Psi(\mathbf{T}_n)$ in \eqref{approx_matrix} to complete the proof.
\end{proof}
\end{thm}
Now, we focus on finding more natural bound that does not use Bessel functions, which is actually studied through the research conducted on the Bessel bounds. To begin, we need the following lemmas.
\begin{lem}[\cite{Yudell}]\label{lemYud}
For $x \in \mathbb{R}$ with $x > 0$ and $k > -\frac{1}{2}$, the Bessel functions of the first kind, $I_k(x)$, satisfy the inequality
\[
I_n(x)<\frac{x^k}{2^kk!}\exp(x).
\]
\end{lem}
Clearly, for any nonzero $x \in \mathbb{C}$ and $k \geq 0$, we have
\[
|I_k(x)|\leq I_k(| x |)<\frac{|x|^k}{2^kk!}\exp(|x|).
\]
This means that for a fixed $x \in \mathbb{C}$, as $k$ tends to infinity, $I_k(x)$ tends to $0$. This implies that for sufficiently large values of $n$, as the size of the matrix, the Bessel functions used in $\Psi(\mathbf{T}_n)$, which have very large indices, will be significantly small.
\begin{lem}[\cite{Nasell}]\label{lem_besl}
For $x \in \mathbb{R}$ with $x > 0$ and $k > -\frac{1}{2}$, the Bessel functions of the first kind, $I_k(x)$, satisfy the inequality
\[
I_{k+1}(x)<I_k(x).
\]
\end{lem}
By using Lemma \ref{lem_besl}, for all $x \in \mathbb{C}$ and $m, k \in \mathbb{N}$, if $m > k \geq 0$, then we have:
\[
\vert I_m(x) \vert <  I_k(\vert x \vert).
\]
Now, we can easily state a bound for the $ij$-th element of $\Psi(\mathbf{T}_n)$, as defined in \eqref{approx_matrix}, as follows:
\begin{equation}\label{bound_first_11}
\vert (\Psi(\mathbf{T}_n))_{ij} \vert 
\leq \exp(\mathrm{Re}(b)+2|a|) \left(\frac{|a|^{|i-j|}}{|i-j|!} +  \frac{|a|^{|i+j|}}{|i+j|!}\right). 
\end{equation}
Now, we present the following theorem.
\begin{thm}\label{M-Te_syme}
For symmetric matrix $\mathbf{T}_n = \mathrm{tridiag}(a,b,a)$ and $i,j\in \{1,\ldots,n\}$:
\begin{equation}\label{ineq_123}
\begin{aligned}
\left\vert \big(\exp(\mathbf{T}_n)\big)_{ij} \right\vert &\leq 2e^{\operatorname{Re}(b)}\left(\dfrac{|a|e}{n+1}\right)^{n+1}+ e^{(\operatorname{Re}(b)+2|a|)} \left(\frac{|a|^{|i-j|}}{|i-j|!} +  \frac{|a|^{|i+j|}}{|i+j|!}\right),
\end{aligned}
\end{equation}
and for non-symmetric tridiagonal matrix $\mathbf{A}_n = \mathrm{tridiag}(a,b,c)$ and $i,j\in \{1,\ldots,n\}$:
\begin{equation}\label{non-symmetric-case}
\begin{aligned}
\left\vert \big(\exp(\mathbf{A}_n)\big)_{ij} \right\vert &\leq 2\left|\sqrt{\frac{a}{c}} \right|^{i-j}e^{\operatorname{Re}(b)}\left(\dfrac{|z|e}{n+1}\right)^{n+1}\\
 &+ e^{(\operatorname{Re}(b)+2|z|)} \left|\sqrt{\frac{a}{c}} \right|^{i-j}\left(\frac{|z|^{|i-j|}}{|i-j|!} +  \frac{|z|^{|i+j|}}{|i+j|!}\right).
\end{aligned}
\end{equation}
where $ z = c \sqrt{\frac{a}{c}} $. 
\begin{proof}
By using Theorem \ref{thm:ineq_1first2} and the bound obtained in \eqref{bound_first_11}, the result is at hand.
To establish the inequality \eqref{non-symmetric-case}, we first need use the  following inequality:
\[
\vert\left(\mathbf{F}_n\circ \exp(\mathbf{T}_n) \right)_{ij} \vert =
\left|\sqrt{a/c} \right|^{i-j} \vert\left(\exp(\mathbf{T}_n) \right)_{ij} \vert,
\]
where $ \mathbf{T}_n = \mathrm{tridiag}(z,b,z)$, see Theorem \ref{jotri} and Theorem \ref{Hadam}. Then, using the inequality in \eqref{ineq_123} for $ \mathbf{T}_n$, the proof is easily completed.
\end{proof}
\end{thm}

Next, to clarify our last results, we provide an example for the bound of the elements found in Theorem \ref{M-Te_syme}.

\begin{exa}
Let us consider the matrix from Example \ref{ex1:matrix}, $ \mathbf{T}_n = \mathrm{tridiag}(1, -2, 1) $, with $ n = 100 $ as the size of $ \mathbf{T}_n $. Now, in Figure \ref{bound:12:ex2:trid}, we illustrate $ \left( \exp( \mathbf{T}_n) \right)_{ij} $ for some $ i $ and $ j $, and bound from Theorem \ref{M-Te_syme} in inequality $ \eqref{ineq_123} $. On the left side of Figure \ref{bound:12:ex2:trid}, we show the output for $ i = 1:20 $ and $ j = 1 $, and on the right side, we show it for $ i = 30:70 $ and $ j = 50 $. The results show that the bound from Theorem \ref{M-Te_syme} effectively estimates the exact values of the entries of the matrix $ \exp( \mathbf{T}_n) $.
\begin{figure}[t]
\begin{center}
\includegraphics[width=0.49\textwidth]{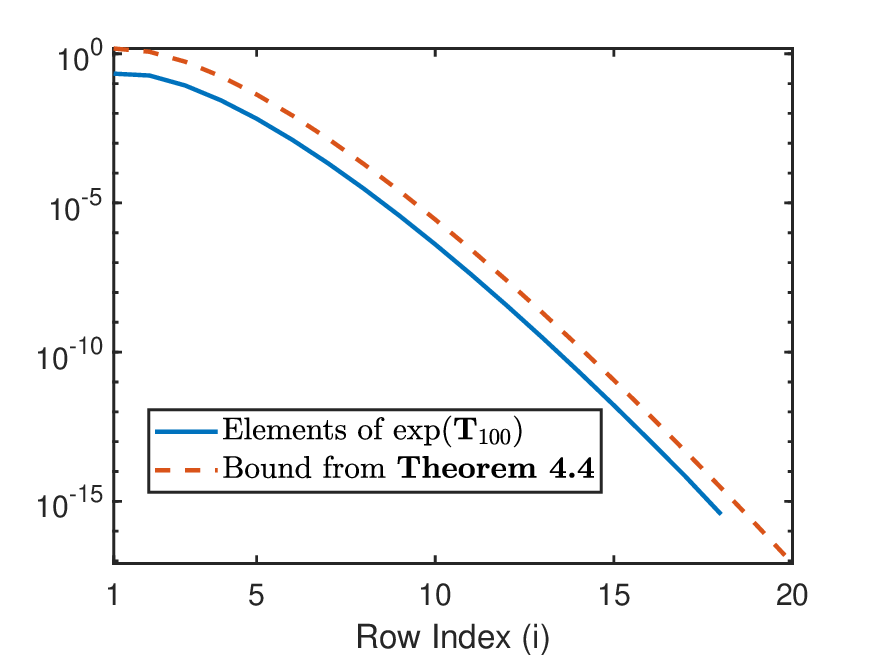}
\includegraphics[width=0.49\textwidth]{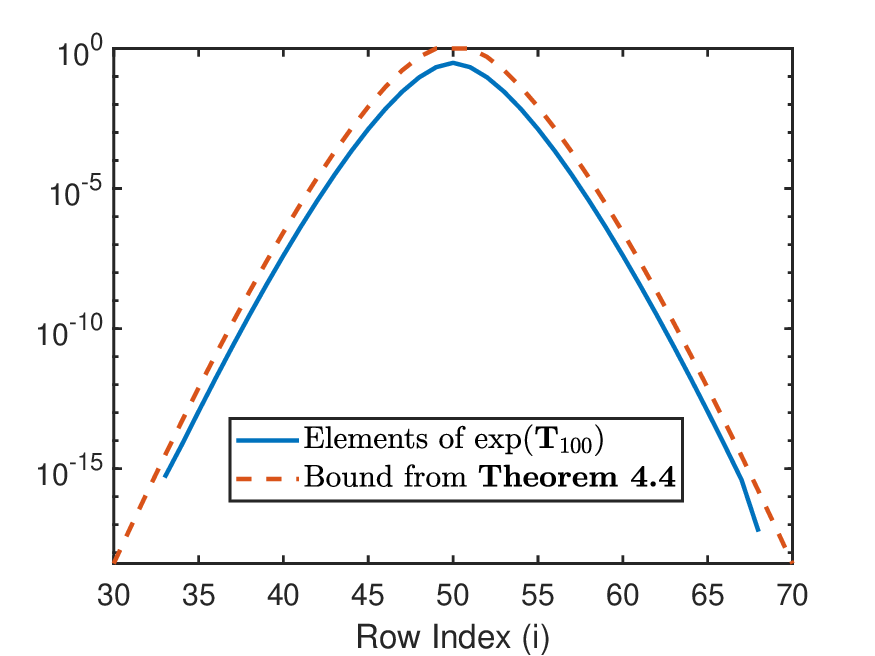}
\caption{\small
Elements \( \left( \exp( \mathbf{T}_n) \right)_{ij} \) and their bound in Theorem \ref{M-Te_syme} for \( n = 100 \), with \( i = 1:20 \) and \( j = 1 \) (left), and \( i = 30:70 \) and \( j = 50 \) (right).}\label{bound:12:ex2:trid}
\end{center}
\end{figure}
\end{exa}
It is worth mentioning that if instead of using the bound in Theorem \ref{M-Te_syme}, the bound in Theorem \ref{thm:ineq_1first2} is used to estimate the elements of exponential of tridiagonal matrices, the accuracy of the estimation will be significantly higher. (see Section \ref{M-V:bound:1})
\subsection{Banded approximation}\label{sec:band:approx}
Theorem \ref{M-Te_syme} not only provides a bound for the entries of the matrix exponential of a tridiagonal matrix but also clearly states that for sufficiently large matrices, the entries of matrix $\exp(\mathbf{T}_n)$ become significantly small as we move away from the main diagonal. This decay depends on the part of the theorem's bound that involves $|i - j|$. Therefore, in practice, we can ignore computing entries that their bounds from Theorem \ref{M-Te_syme} are smaller than a given machine epsilon $\varepsilon>0$. This means that we can approximate $\exp(\mathbf{T}_n)$ with a banded matrix. Now we can approximate $ \exp(\mathbf{T}_n) $ with a $ d $-banded matrix for some $ d>0 $, a matrix in which the elements corresponding to indices $ (i,j) $ with $ |i-j| > d $ are zero. This fact helps us make our approach more efficient. 

For example, set $\varepsilon = 10^{-16}$ and $\mathbf{T}_n = \mathrm{tridiag}(1,-2,1)$ with $n = 100$, the entries of the corresponding Toeplitz and Hankel matrices in \( \Psi(\mathbf{T}_n) \) with magnitudes larger than $\varepsilon$ are illustrated in Figure \ref{spy}. This figure shows that we can approximate $\mathbf{T}_n$ with a 19-banded matrix. Here, we approximate the matrix $\exp(\mathbf{T}_n)$ with a $d$-banded matrix, which we denote as $\Psi^{[d]}(\mathbf{T}_n)$.

\begin{figure}
\begin{center}
\includegraphics[width=4.3cm, height=4.3cm,]{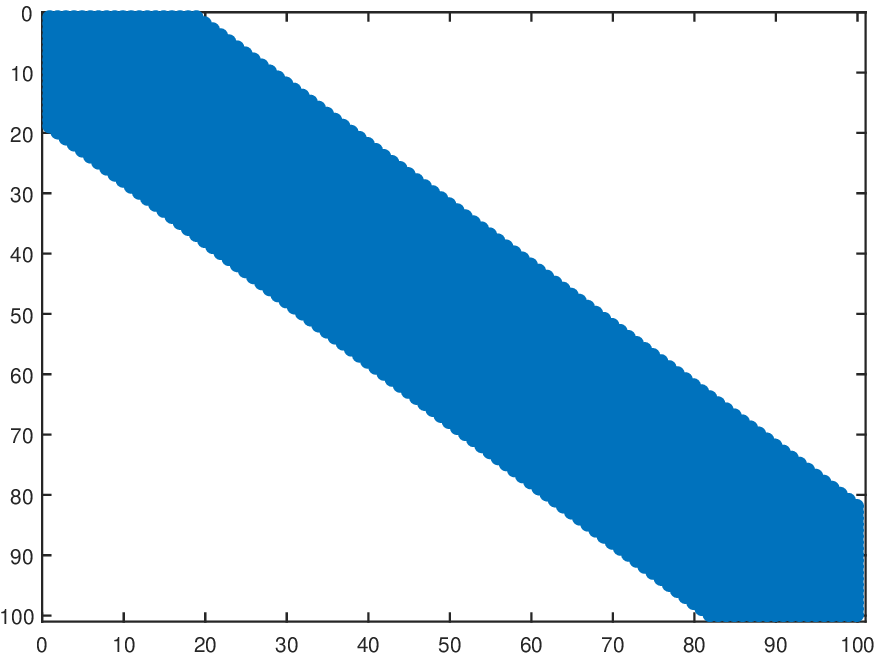}
\hspace*{1cm}
\includegraphics[width=4.3cm, height=4.3cm,]{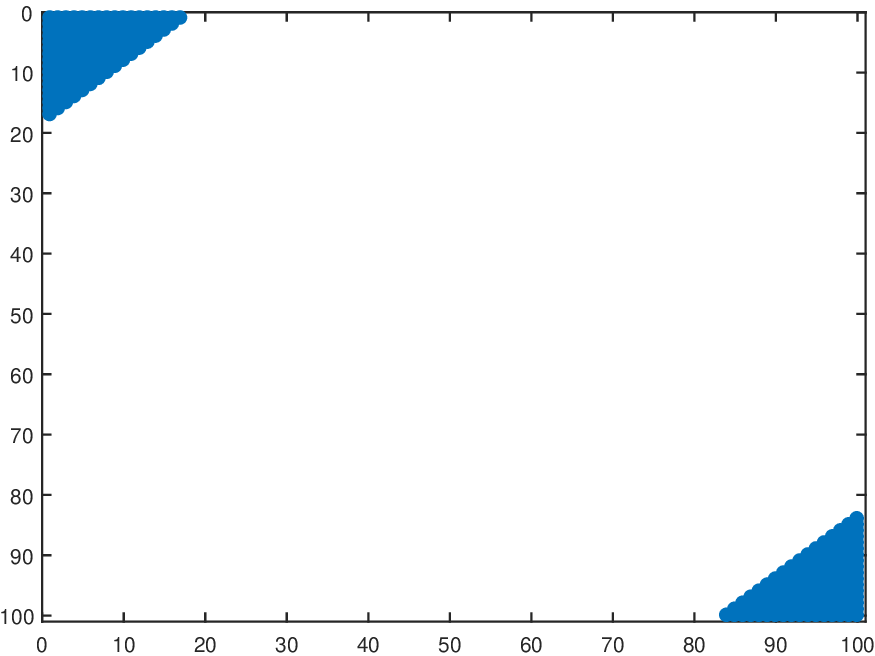}
\caption{Sparsity pattern of the 100-by-100 Toeplitz (left) and Hankel (right) matrices with elements having magnitudes larger than \( \varepsilon = 10^{-16} \), where these matrices are used in \( \Psi(\mathbf{T}_n) \) for \( \mathbf{T}_n = \mathrm{tridiag}(1,-2,1) \).}
\label{spy}
\end{center}
\end{figure}

For a given $d > 0$, let us define $\Psi^{[d]}(\mathbf{T}_n)$, where $\mathbf{T}_n = \mathrm{tridiag}(a, b, a)$. By using only modified Bessel functions of the first kind of index $0$ to $d+2$, i.e., $I_0(2a), I_1(2a), \ldots$, $I_d(2a), I_{d+1}(2a), I_{d+2}(2a)$, we define vectors of length $n$ as follows:
\begin{align*}
\mathbf{u}^{[d]}&=( e^bI_0(2a),e^bI_1(2a),\ldots,e^bI_{d-1}(2a),e^bI_{d}(2a),0,\ldots,0)\\
\mathbf{v}^{[d]}&=(e^bI_2(2a),\ldots,e^bI_{d+1}(2a),e^bI_{d+2}(2a),0,\ldots,0),
\end{align*}
such that $\Psi^{[d]}(\mathbf{T}_n)$ is given by the difference of a Toeplitz and a Hankel matrix:
\begin{equation}\label{banded:Psi:new}
\Psi^{[d]}(\mathbf{T}_n)=\mathrm{Toeplitz}\left\lbrace\mathbf{u}^{[d]}\right\rbrace-\mathrm{Hankel}\left\lbrace\mathbf{v}^{[d]}\right\rbrace.
\end{equation}
Here, $\Psi^{[d]}(\mathbf{T}_n)$ is a $d$-banded matrix as an approximation of $\exp(\mathbf{T}_n)$. Additionally, we can approximate $\exp(\mathbf{A}_n)$, where $\mathbf{A}_n = \mathrm{tridiag}(a, b, c)$, by using Theorem \ref{jotri}. First, by defining
$\mathbf{F}_n^{[d]} = \mathrm{Toeplitz}\left\lbrace\mathbf{x}^{[d]},\mathbf{y}^{[d]}\right\rbrace$, where 
\[
\mathbf{x}^{[d]} = (a_0, a_1, \ldots, a_{d},\underbrace{0,\ldots,0}_{\small n-d-1~\text{times}}),~~~~~~~~~\mathbf{y}^{[d]} = (a_0, a_{-1}, \ldots, a_{-d},\underbrace{0,\ldots,0}_{\small n-d-1~\text{times}}),
\]
with $a_{k} = \left( \sqrt{\frac{a}{c}} \right)^{k}$. So, we can approximate the exponential of non-symmetric tridiagonal matrices as follows:
\begin{equation}\label{band:non:sym:approx}
\exp(\mathbf{A}_n) \approx \mathbf{F}_n^{[d]} \circ \Psi^{[d]}(\mathbf{T}_n),
\end{equation}
where appropriate  $\mathbf{T}_n$ is defined in Theorem \ref{Hadam}.

Note that in $\Psi^{[d]}(\mathbf{T}_n)$, we only need to compute $d+3$ modified Bessel functions at $2a$ using the trapezoidal rule with $m$ quadrature points. Therefore, the complexity of the method is $\mathcal{O}(\max\{m,d\} \cdot d)$. Now, to obtain the error bound for this approximation, we estimate it as follows.
\begin{equation}\label{bound:error:d:band}
\begin{aligned}
\Vert \exp(\mathbf{T}_n) - \Psi^{[d]}(\mathbf{T}_n) \Vert_{\infty} &\leq \Vert \exp(\mathbf{T}_n) - \Psi(\mathbf{T}_n) \Vert_{\infty} +\Vert \Psi(\mathbf{T}_n) - \Psi^{[d]}(\mathbf{T}_n) \Vert_{\infty}\\
&\leq 2e^{\operatorname{Re}(b)}\left(\dfrac{|a|e}{n+1}\right)^{n+1} + \frac{4|a|^{d+1}}{(d+1)!} \exp(\operatorname{Re}(b)+3|a|).
\end{aligned}
\end{equation}
The bound \eqref{bound:error:d:band} easily obtain by utilizing Corollary \ref{approx}, Lemma \ref{lemYud}, and Lemma \ref{lem_besl}.
 Also, for \( \mathbf{A}_n = \mathrm{tridiag}(a,b,c) \), in the non-symmetric case, only \( \Delta \) as used in Theorem \ref{Hadam} is multiplied in the above bound, and \( a \) is replaced by \( z = c \sqrt{\frac{a}{c}} \). In Section \ref{sec:example:63}, we examined the effectiveness of bound in \( \eqref{bound:error:d:band} \) and tested it for different matrices.
\subsection{Fast matrix-vector product}\label{FFT:Section}
 The matrix-vector product with a Toeplitz matrix can be computed using the Fast Fourier Transform (FFT) in $\mathcal{O}(n \log(n))$ operations \cite{Bai-Demmel-Dongarra}. We note that $\Psi(\mathbf{T}_n)$ is a Toeplitz plus Hankel matrix. Therefore, for vector $\mathbf{v}$ the computation of $\Psi(\mathbf{T}_n)\mathbf{v}$ can be performed efficiently. The Toeplitz part can be handled using the FFT in $\mathcal{O}(n \log(n))$ operations. Similarly, the Hankel part can also be treated efficiently by decomposing it into the product of a Toeplitz matrix and a backward identity matrix. Using this decomposition, the multiplication of a vector by a Hankel matrix can also be carried out in $\mathcal{O}(n \log(n))$ operations. Thus, computing $\Psi(\mathbf{T}_n)\mathbf{v}$ requires only $\mathcal{O}(2n \log(n))$ operations.
 
\section{Approximation of the exponential of tridiagonal matrix and the heat equation}\label{Method}
In this section, using the studies conducted in the previous sections on tridiagonal matrices, we present an approximation of solution to the one-dimensional and two-dimensional heat equation that leads to a stable method.
\subsection{One-dimensional heat equation}\label{sec:heat:eq1}
Let us consider the one-dimensional heat equation as follows:
\begin{equation}\label{h1}
u_t = au_{xx} , \qquad x\in [b,c] , \qquad t>0 , \qquad a>0,
\end{equation}
with initial condition $u(x,0) = u_0(x)$ for $x \in [b,c]$, and homogeneous Dirichlet boundary conditions $u(b,t) = u(c,t) = 0$ for $t>0$.
Consider points $x_j$ in $[b,c]$ with $x_j = b + j \Delta x$, for $j = 1, \dots, n$, where $\Delta x = \frac{c-b}{n+1}$, and $n \in \mathbb{N}$.  By applying the central difference discretization of the second-order spatial derivative for $u_{xx}$ in \eqref{h1}, the equation can now be written as follows:
\[ \dfrac{ \mathrm{dU}_{j}}{\mathrm{dt}}= a\dfrac{\mathrm{U}_{j+1}-2\mathrm{U}_j + \mathrm{U}_{j-1}}{(\mathrm{\Delta} x)^2}, \quad j= 1,\ldots , n,\]
where $\mathrm{U}_{j}(t)$ is:
\begin{equation}\label{uj:initial:mjhjh}
\mathrm{U}_{j}(t)\approx u(x_j, t)\ ,\quad j= 1,\ldots , n.
\end{equation}
Now, the heat equation \eqref{h1} is transformed into the following system of ordinary differential equations:
\begin{equation}\label{eq:transf:heat}
\dfrac{\mathrm{d}\mathbf{U}}{\mathrm{dt}}=\mathbf{K}_{n}\mathbf{U},
\end{equation}
where $\mathbf{U}(t) = \left(\mathrm{U}_{1}(t),\ldots,\mathrm{U}_{n}(t)\right)^T$ that $\mathrm{U}_j(t)$ in \eqref{uj:initial:mjhjh}, and $\mathbf{K}_n= \mathrm{tridiag}(\frac{a}{(\Delta x)^2},\frac{-2a}{(\Delta x)^2},\frac{a}{(\Delta x)^2})$ in $\mathbb{R}^{n \times n}$.
Solution of the system in \eqref{eq:transf:heat} is:
\[{\mathbf{U}}=\exp(\mathbf{K}_{n}\mathrm{t})\mathbf{U}_0,\]
where $\mathbf{U}_0 = ( u_0(x_1),\ldots, u_0(x_n))^T$. Moreover, by choosing a time step size $\Delta t$ and defining $\mathbf{U}^{(k)} := \mathbf{U}(k \Delta t)$, the method can be written as:
\begin{equation}\label{eq:k:th:equa}
\mathbf{U}^{(k)}=\exp(\mathbf{T}_{n})\mathbf{U}^{(k-1)},
\end{equation}
where $\mathbf{T}_{n}=\text{tridiag}(\mu,-2\mu,\mu)$ and $\mu=\frac{a\Delta t}{(\Delta x)^2}$. As discussed in the previous sections, we can approximate $\exp(\mathbf{T}_{n})$ by $\Psi(\mathbf{T}_{n})$, and use FFT to compute $\Psi(\mathbf{T}_{n}) \mathbf{U}^{(k-1)}$ at each step with $\mathcal{O}(n \log(n))$ complexity, as detailed in Section~\ref{FFT:Section}.

Also, to reduce computational complexity, we can use the approximation $\Psi^{[d]}(\mathbf{T}_{n})$ from Section \ref{sec:decay}, in \eqref{banded:Psi:new}. The corresponding method to \eqref{eq:k:th:equa} for a suitable value of $d>0$ is:
\begin{equation} \label{num}
\mathbf{U}^{(k,d)}=\Psi^{[d]}(\mathbf{T}_{n})\mathbf{U}^{(k-1,d)},
\end{equation}
where $\Psi^{[d]}(\mathbf{T}_{n})$ is used instead of $\exp(\mathbf{T}_{n})$. The method presented in \eqref{num} provides an approximation of the solution to the heat equation in \eqref{h1} through its output, the vector $\mathbf{U}^{(k,d)}$. One advantage of method \eqref{num} is that it unconditionally preserves the maximum principle, which we aim to prove in the following.

Let $\Vert \cdot \Vert_{\infty}$ denote the maximum absolute value of elements in a vector. The following theorem establishes the stability of the numerical method \eqref{num}.

\begin{thm}\label{thm:stable:max}
Let $\mathbf{U}^{(k,d)}$ be the vector obtained from method \eqref{num} at the $k$-th step, and let $\mathbf{U}_0$ be the initial vector. They satisfy the following inequality:
\begin{equation}\label{max}
\|\mathbf{U}^{(k,d)}\|_{\infty}<\|\mathbf{U}_0\|_{\infty}.\nonumber
\end{equation}
In other words, the method in \eqref{num} preserves the maximum principle unconditionally.
\begin{proof}
First, from \eqref{num}, we have  
\[
\mathbf{U}^{(k,d)} = \left(\Psi^{[d]}(\mathbf{T}_{n})\right)^k \mathbf{U}_0.
\]  
Since $\Vert \Psi^{[d]}(\mathbf{T}_{n}) \Vert_{\infty} \leq \| \Psi(\mathbf{T}_{n}) \|_{\infty}$, we proceed with the proof to find a bound for $\| \Psi(\mathbf{T}_{n}) \|_{\infty}$. On the other hand, since the Bessel function satisfies $I_m(\cdot) = I_{-m}(\cdot)$ for integer values of $m$, and since $(\Psi(\mathbf{T}_n))_{ij} = \exp(-2\mu)\left(I_{i-j}(2\mu) - I_{i+j}(2\mu)\right)$, using the fact that the Bessel function is decreasing with respect to its index for positive arguments, i.e., $I_{m+1}(2\mu) \leq I_m(2\mu)$, we can conclude that $(\Psi(\mathbf{T}_n))_{ij} < \exp(-2\mu) I_{|i-j|}(2\mu)$. In this case, since for each row of $\Psi(\mathbf{T}_n)$ the term $\exp(-2\mu) I_{|i-j|}(2\mu)$ includes exactly one instance where the Bessel function has index $0$, and at most two instances with indices ranging from $1$ to $n-1$, we can say that  
\[
\| \Psi(\mathbf{T}_{n}) \|_{\infty} \leq \exp(-2\mu)\left(I_0(2\mu) + 2\sum_{p=1}^{n-1} I_p(2\mu)\right),
\]
and we have the property that the modified Bessel functions of the first kind satisfy in \cite{Arfken}:
\[
\exp(2\mu)=I_0(2\mu)+2\sum_{p=1}^{\infty}I_p(2\mu).
\]
Thus $\| \Psi(\mathbf{T}_{n}) \|_{\infty} <1$, and this completes the proof.
\end{proof}
\end{thm}





\subsection{ Two-dimensional heat equation}
Consider the two-dimensional heat equation:
\begin{equation}\label{h2}
u_t=a\nabla^2{ u},~~~\mathbf{x}=(x,y)\in \Omega\subset\mathbb{R}^2,~t>0,~~a>0,
\end{equation}
with the given boundary and initial conditions
\[u({\bf x},t)=0,~~~{\bf x}\in \partial {\Omega },~t>0,~~~~~~~~u({\bf x},0)=u_0({\bf x}),~~~{\bf x}\in \Omega.\]
Consider a uniform rectangular grid of points $(x_i, y_j)$ in the region $\Omega = (b,c) \times (d,e)$, where
$x_i = b + i\Delta x,$ for $i = 1, \dots, n_x,$ and
$y_j = d + j\Delta y,$ for $j = 1, \dots, n_y.$
Here, the discretization parameters in the $x$ and $y$ directions are given by
$\Delta x = \frac{c - b}{n_x + 1}$ and $\Delta y = \frac{e - d}{n_y + 1},$
with $n_x, n_y \in \mathbb{N}$.
Spatial discretization of (\ref{h2}) using central second-order formulas is performed for $i = 1, \ldots, n_x$ and $j = 1, \ldots, n_y$, as follows:
\[
\dfrac{\mathrm{dU}_{i,j}}{\mathrm{dt}}= a \dfrac{\mathrm{U}_{i+1,j}-2\mathrm{U}_{i,j} + \mathrm{U}_{i-1,j}}{(\Delta x)^2}+a\dfrac{\mathrm{U}_{i,j+1}-2\mathrm{U}_{i,j} + \mathrm{U}_{i,j-1}}{(\Delta y)^2}.
\]
where the approximate solution is denoted by $\mathrm{U}_{i,j}(t)$, and
\[
\mathrm{U}_{i,j}(t)\approx u(x_i, y_j, t)\ ,\quad i= 1,\ldots , n_x,~~ j= 1, \ldots , n_y,
\]
that here, time remains continuous.
We can now transform the heat equation to \cite{Quarteroni}:
\begin{equation}\label{sys2}
\dfrac{\mathrm{d}\mathbf{U}}{\mathrm{dt}}= \mathbf{A}_n\mathbf{U},
\end{equation}
where $n = n_x n_y$ and $\mathbf{A}_n \in \mathbb{R}^{n \times n}$, which the matrix $\mathbf{A}_n$ has block tridiagonal of the form 
$ \mathbf{A}_n = \mathrm{tridiag}(\mathbf{D}_{n_x}, \mathbf{S}_{n_x}, \mathbf{D}_{n_x}) $.
The matrix $\mathbf{D}_{n_x} \in \mathbb{R}^{n_x \times n_x}$ is diagonal, with diagonal entries $\frac{a}{(\Delta y)^2}$, while $\mathbf{S}_{n_x} \in \mathbb{R}^{n_x \times n_x}$ is the following symmetric tridiagonal matrix:
\[
\mathbf{S}_{n_x} = \mathrm{tridiag}\left(\frac{a}{(\Delta x)^2},-\frac{2a}{(\Delta x)^2}-\frac{2a}{(\Delta y)^2} ,\frac{a}{(\Delta x)^2}\right).
\]
Now, let us define the tridiagonal matrix
$ \mathbf{H}_{n_y} = \mathrm{tridiag}(1,0,1) \in \mathbb{R}^{n_y \times n_y}, $
then $\mathbf{A}_n$ can be written as:
\begin{equation}\label{Kroncker:pro:matrix}
\mathbf{A}_n =  \dfrac{a}{(\Delta y)^2}\mathbf{H}_{n_y} \otimes \mathbf{I}_{n_x}+\mathbf{I}_{n_y} \otimes \mathbf{S}_{n_x},
\end{equation}
where $\mathbf{I}_{n_i}$ is the identity matrix, and $\otimes$ is the Kronecker product. More studies on the exponential of the matrix $\mathbf{A}_n$ in \eqref{Kroncker:pro:matrix} have been conducted in \cite{Benzi-Simoncini}. The solution of system \eqref{sys2} is:
\[
\mathbf{U} = \left(\exp\left[\dfrac{a}{(\Delta y)^2}\mathbf{H}_{n_y}\mathrm{t}\right]\otimes \exp\left[\mathbf{S}_{n_x}\mathrm{t}\right]\right)\mathbf{U}_0.
\]
Let $\mu_x=\frac{a \Delta t}{(\Delta x)^2}$, $\mu_y=\frac{a \Delta t}{(\Delta y)^2}$, $\mathbf{T}_{n_x}=\text{tridiag}(\mu_x,-2\mu_x-2\mu_y,\mu_x)$, $\mathbf{K}_{n_y}=\text{tridiag}(\mu_y,0,\mu_y)$, and by choosing a time step size $\Delta t$, the numerical method is:
\[\mathbf{U}^{(k+1)} = (\exp[\mathbf{K}_{n_y}]\otimes \exp[\mathbf{T}_{n_x}])\mathbf{U}^{(k)}.\]
where $\mathbf{U}^{(k)} := \mathbf{U}(k \Delta t)$. More precisely, a splitting based on the Kronecker product is implemented for solving the two-dimensional heat equation. Due to the approximation method presented in Section \ref{ESTTM}, for suitable values of $d_x$ and $d_y$, the following numerical method can be considered:
\begin{equation}\label{app2}
\mathbf{U}^{(k+1,d_x,d_y)} = \left(\Psi^{[d_y]}(\mathbf{K}_{n_y})\otimes \Psi^{[d_x]}(\mathbf{T}_{n_x})\right)\mathbf{U}^{(k,d_x,d_y)}.
\end{equation}

Similar to Theorem \ref{thm:stable:max}, we can prove that:
\begin{align*}
\left\Vert \Psi^{[d_y]}(\mathbf{K}_{n_y})\otimes \Psi^{[d_x]}(\mathbf{T}_{n_x}) \right\Vert_{\infty}< 1,
\end{align*}
which shows the stability of the method in \eqref{app2} and that the method unconditionally preserves the maximum principle.
\section{Numerical examples}\label{Nexamples}
In this section, we present various numerical experiments to investigate quality of our methods, approximations, and error bounds. All experiments were conducted in MATLAB \texttt{R2021a}. We use the command ``$\mathtt{besseli}$'' to compute the values of the modified Bessel functions of the first kind. We compute quantities $\exp(A)$ and $(\exp(A))_{pl}$ by using the MATLAB command ``$\mathtt{expm}$'' to obtain the error of our approximations to machine precision.

\subsection{Iserles' bound for the entries of the exponential of tridiagonal matrices} 
Iserles, in \cite[Theorem 2.2]{Iserles}, established a bound for the entries of the exponential of tridiagonal matrices. To derive this bound, consider a tridiagonal matrix $A $ and define $\rho := \max_{i,j} |(A)_{ij}|$. By using modified Bessel functions of the first kind, $I_k(\cdot)$, he proved the following bound: 
\begin{equation}\tag{Iserles' bound} 
\vert (\exp(A))_{ij} \vert \leq \exp(\rho) I_{|i-j|}(2\rho). 
\end{equation} 
In this example, we aim to compare the bound in Theorem \ref{M-Te_syme} with Iserles' bound.
\begin{figure}[t]
\begin{center}
\includegraphics[width=0.49\textwidth]{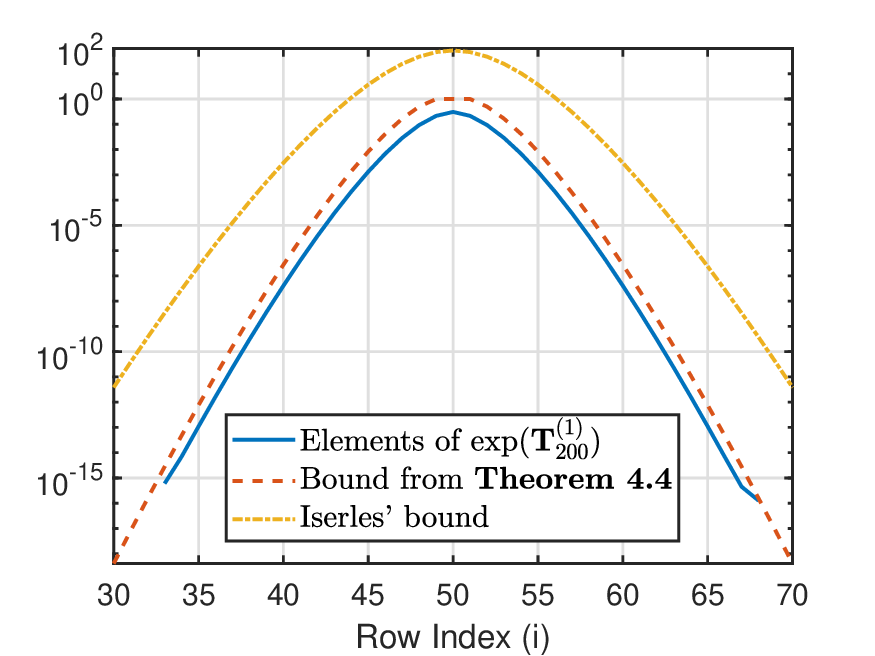}
\includegraphics[width=0.49\textwidth]{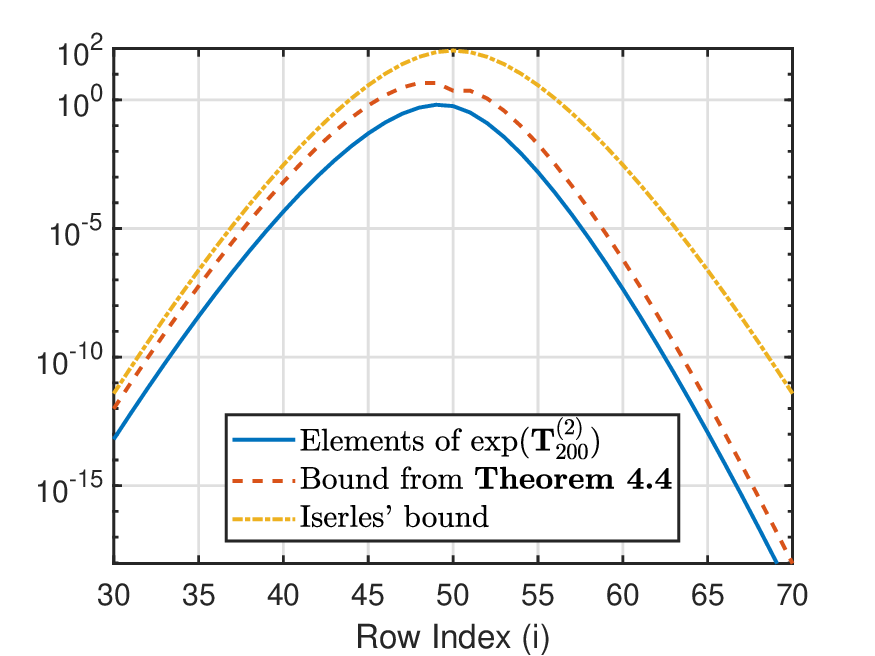}
\caption{Elements $|( \exp( \mathbf{T}_n))_{30:70,50}|$ and their Iserles' bounds, along with bounds from Theorem \ref{M-Te_syme} for  $\mathbf{T}^{(1)}_{200} = \mathrm{tridiag}(1,-2,1)$ (left) and for  $\mathbf{T}^{(2)}_{200} = \mathrm{tridiag}(1,-2,2)$ (right).}\label{bound:1_2:Fig2}
\end{center}
\end{figure}

In this experiment we consider two matrices, $\mathbf{T}^{(1)}_{200} = \mathrm{tridiag}(1,-2,1)$ and matrix $\mathbf{T}^{(2)}_{200} = \mathrm{tridiag}(1,-2,2)$. In Figure \ref{bound:1_2:Fig2}, we demonstrate the elements $(\exp(\mathbf{T}_{200}^{(r)}))_{30:70,50}$ and their bounds from Theorem \ref{M-Te_syme}, along with Iserles' bound. The result tells us that the bound from Theorem \ref{M-Te_syme} works better in estimating the elements. One of the reasons why the bound from the theorem works better may be that, in this bound, the sign of the elements of the main diagonal and the elements off the main diagonal, even if they have smaller absolute values, affect the bound, while in Iserles' bound, only the maximum absolute value of the elements of the matrix is considered.

\subsection{Bound for the entries of the exponential of a banded Hermitian matrix}\label{M-V:bound:1}
In \cite[Section 4]{Benzi-Simoncini}, the authors provided a bound for the entries of $\exp(-\tau M)$, where $M$ is a Hermitian positive semi-definite matrix. To recall the result, let $M$ have eigenvalues in the interval $[0, 4\rho]$, and for indices $i \neq j$, let $\xi = \lceil |i-j|/\beta \rceil$, where $ \beta $ is half band width of $ M $.  Then, the bound established in \cite{Benzi-Simoncini}, called the ``M.V. bound'', is given by:
\begin{enumerate}
\item[i)] For $\rho\tau\ge 1$ and $\sqrt{4\rho\tau}\le \xi \le 2\rho\tau$,
\[
| (\exp(-\tau M) )_{ij}|
\le
10 \exp\left(-\frac{\xi^2}{5 \rho\tau}\right).
\]
\item[ii)]
For $\xi  \ge 2\rho\tau$,
\[
| (\exp(-\tau M) )_{ij}|
\le
10 \frac{\exp(-\rho\tau)}{\rho\tau}
\left ( \frac{{\rm e}\rho\tau}{\xi}\right)^{\xi }.
\]
\end{enumerate}

\begin{figure}[t] 
     \centering
     \begin{subfigure}[b]{0.49\textwidth}
         \centering
         \includegraphics[width=\textwidth]{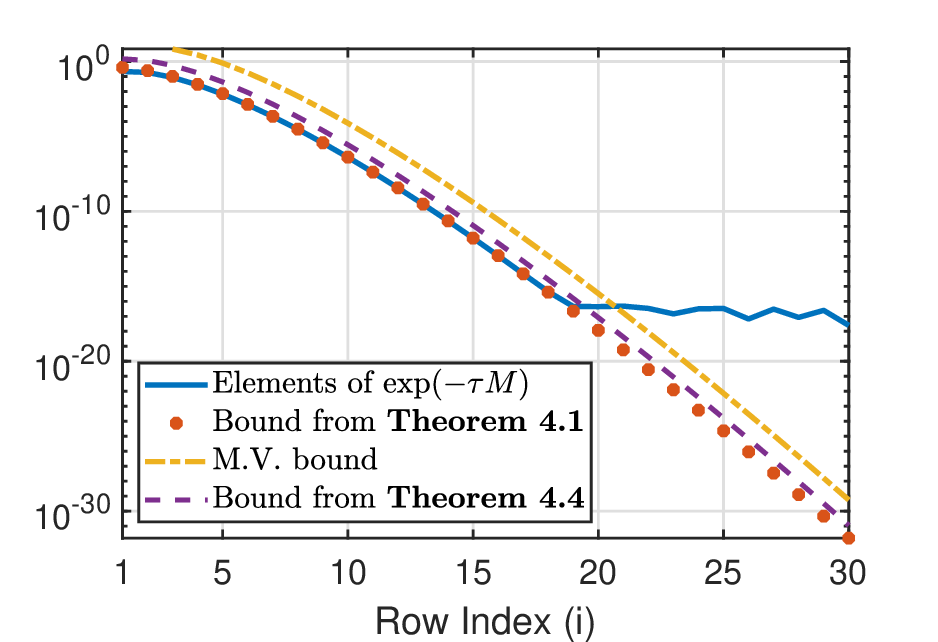}
\caption{$\tau = 1$.}
     \end{subfigure}
     \hfill
     \begin{subfigure}[b]{0.49\textwidth}
         \centering
         \includegraphics[width=\textwidth]{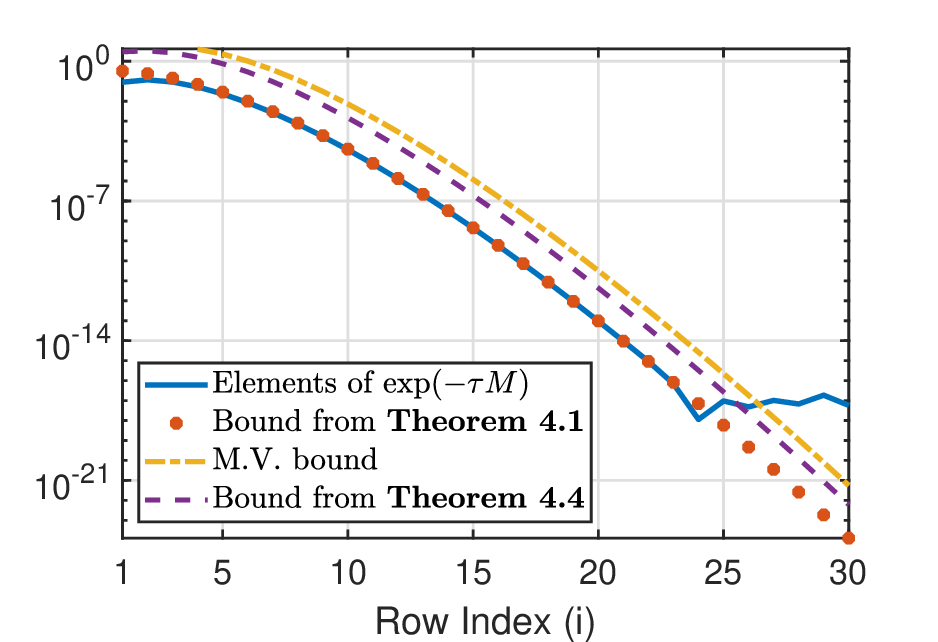}
\caption{ $\tau = 2$.}
     \end{subfigure}
     \hfill
          \begin{subfigure}[b]{0.49\textwidth}
         \centering
         \includegraphics[width=\textwidth]{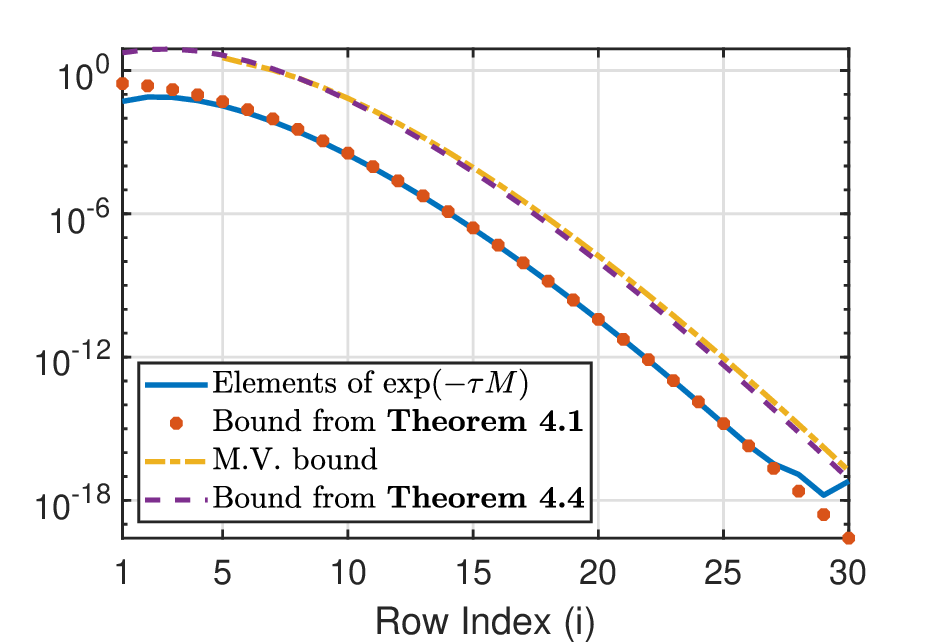}
\caption{$\tau = 3$.}
     \end{subfigure}
     \hfill
     \begin{subfigure}[b]{0.49\textwidth}
         \centering
         \includegraphics[width=\textwidth]{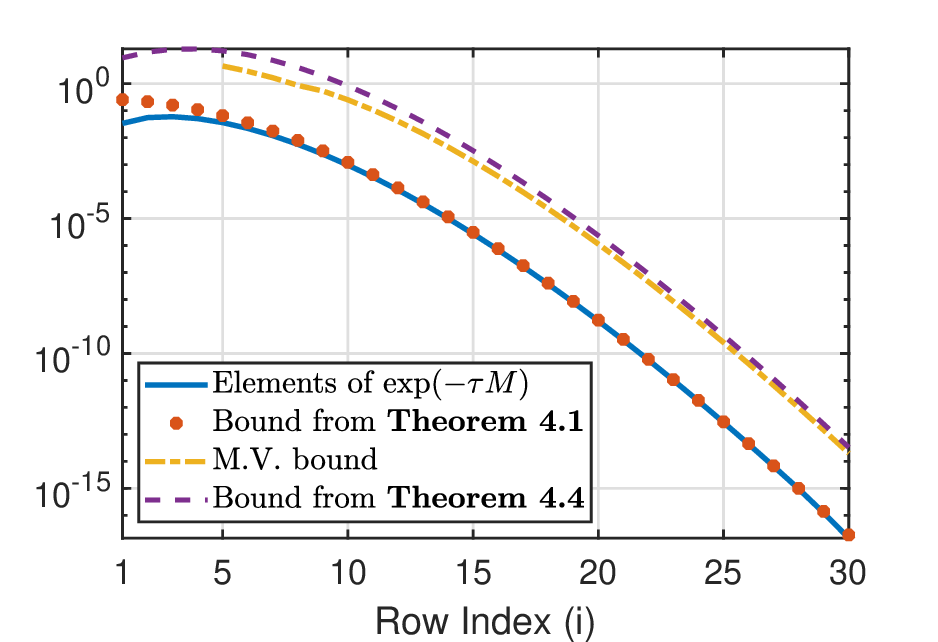}
\caption{ $\tau = 4$.}
     \end{subfigure}
        \caption{Elements $|(\exp(-\tau M))_{1:30,1}|$ and their bounds from Theorem \ref{M-Te_syme} and Theorem \ref{thm:ineq_1first2}, along with the M.V. bound for $\tau \in \{1,2,3,4\}$, and $M = M_1 - \lambda_{\min}(M_1)I$, where $M_1 = \mathrm{tridiag}(-1,4,-1)\in \mathbb{R}^{200\times 200}$ .}
 \label{Fig-T:Bound}
 \end{figure}
 In Figure \ref{Fig-T:Bound}, we illustrate elements $|(\exp(-\tau M))_{1:30,1}|$, where $M\in \mathbb{R}^{200\times 200}$, and their bounds from Theorem \ref{M-Te_syme}, from Theorem \ref{thm:ineq_1first2}, and M.V. bound for $\tau\in\{1,2,3,4\}$. In this example, the matrix $M$ is as follows: first, we take $M_1 = \mathrm{tridiag}(-1,4,-1)$, and then $M = M_1 - \lambda_{\min}(M_1) I$, where $ \lambda_{\min}(M_1) $ is the smallest eigenvalue of $ M_1 $. Finally, we set $\rho = \left(\lambda_{\max}(M_1)-\lambda_{\min}(M_1)\right)/4$, as done in \cite[Example 4.3]{Benzi-Simoncini}, and since $M$ is a tridiagonal matrix, we set $\beta = 1$. The result tells us that the bound from Theorem \ref{thm:ineq_1first2} estimates the entries better than other bounds, and the bound from Theorem \ref{M-Te_syme} performs better than the M.V. bound in most cases. However, we must mention that the bound from Theorem \ref{thm:ineq_1first2} estimates the entries with high accuracy. As we can see, for different original matrices, the high accuracy of the estimation  from Theorem \ref{thm:ineq_1first2} is preserved, which can be considered an advantage of this bound.

\subsection{Error bound of the approximation of the exponential of tridiagonal matrix}\label{sec:example:63}
In Section \ref{error:analysis:new} of this paper, in Corollary \ref{approx}, we provided a bound for $\Vert \exp(\mathbf{T}_n) - \Psi(\mathbf{T}_n) \Vert_{\infty}$, where $\Psi(\mathbf{T}_n)$ is the approximation method presented in \eqref{approx_matrix}. We also introduced another approximation method in \eqref{banded:Psi:new}, where we approximate $\exp(\mathbf{T}_n)$ by the $d$-banded matrix $\Psi^{[d]}(\mathbf{T}_n)$, with the corresponding error bound given in \eqref{bound:error:d:band}. In this experiment, we aim to test these approximation methods and the established error bounds. In this example, we consider $\mathbf{T}_n = \mathrm{tridiag}(1,-2,1) \in \mathbb{R}^{n\times n}$, which arises in the numerical solution of the heat equation.

\begin{figure}[t] 
     \centering
     \begin{subfigure}[b]{0.49\textwidth}
         \centering
         \includegraphics[width=\textwidth]{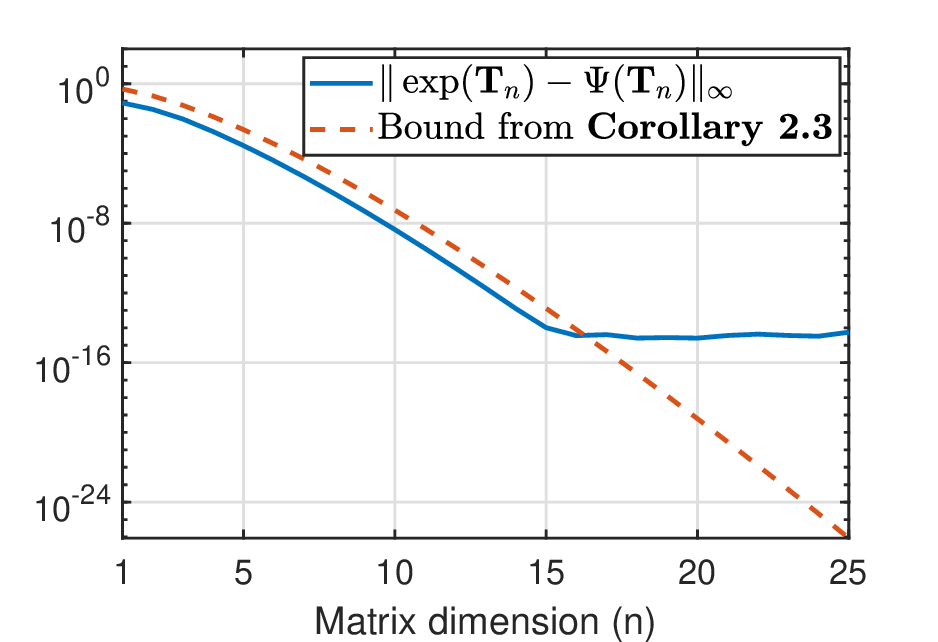}
\caption{$n = 1:25$.}\label{a-Fig63}
     \end{subfigure}
     \hfill
     \begin{subfigure}[b]{0.49\textwidth}
         \centering
         \includegraphics[width=\textwidth]{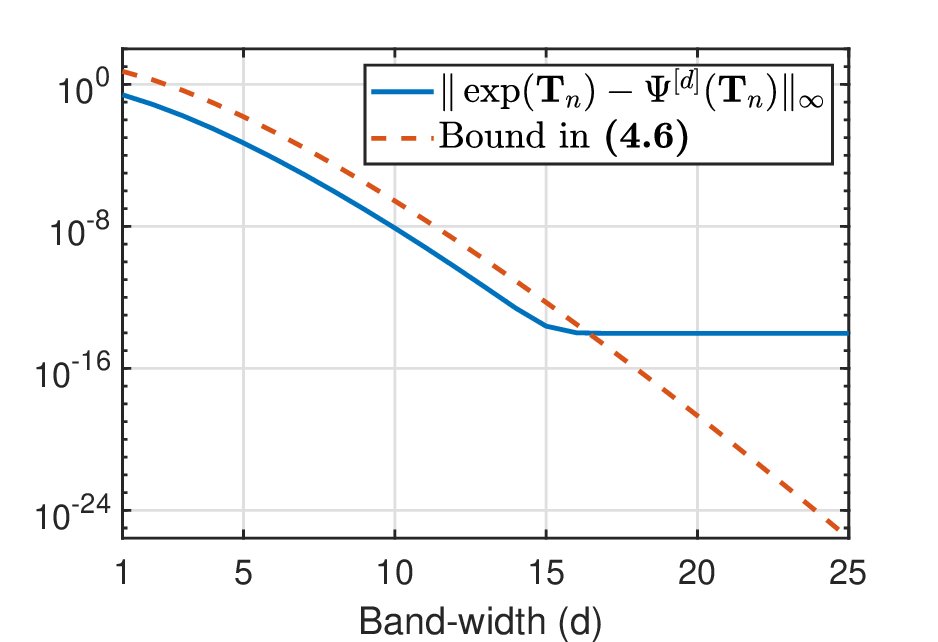}
\caption{ $n = 200$ and $d = 1:25$.}
     \end{subfigure}
     \hfill
          \begin{subfigure}[b]{0.49\textwidth}
         \centering
         \includegraphics[width=\textwidth]{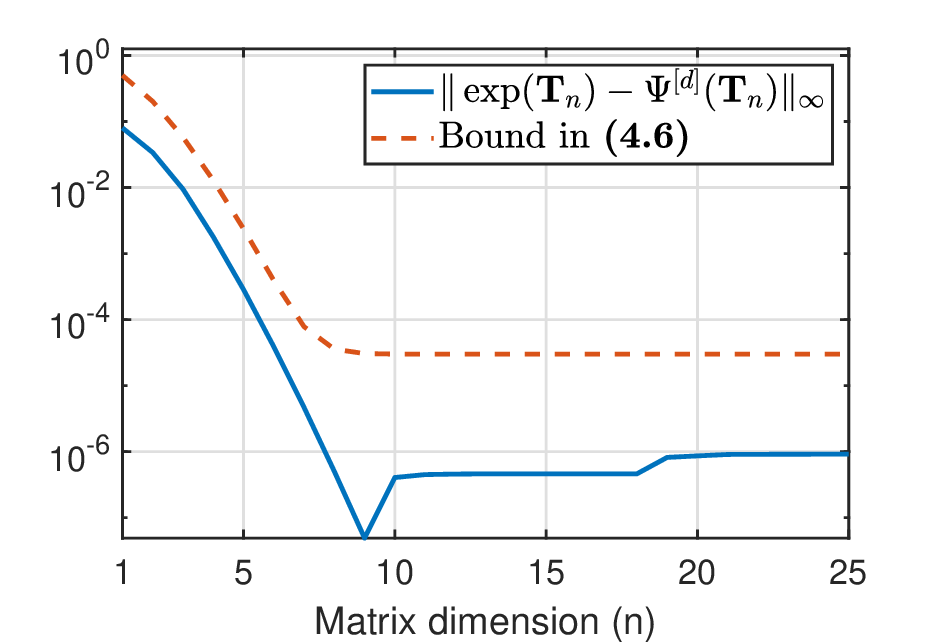}
\caption{$n = 1:25$ and $d = 8$.}
     \end{subfigure}
     \hfill
     \begin{subfigure}[b]{0.49\textwidth}
         \centering
         \includegraphics[width=\textwidth]{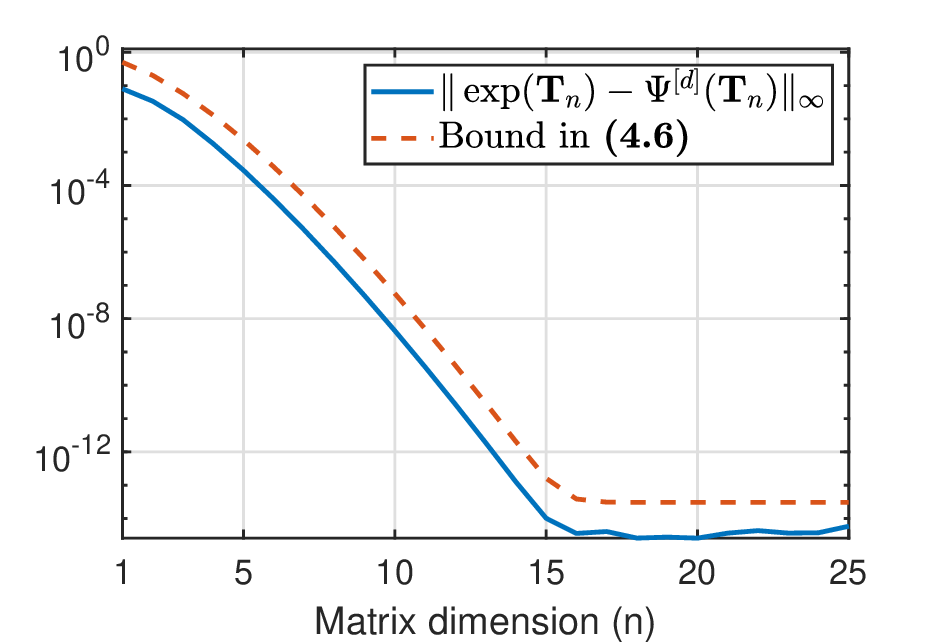}
\caption{ $n = 1:25$ and $d = 16$.}
     \end{subfigure}
\caption{Comparison of $\Vert \exp(\mathbf{T}_n) - \Psi(\mathbf{T}_n) \Vert_{\infty}$ with the bound from Corollary \ref{approx} (a), and for the other figures, comparison of $\Vert \exp(\mathbf{T}_n) - \Psi^{[d]}(\mathbf{T}_n) \Vert_{\infty}$ with the bound in \eqref{bound:error:d:band}, where $\mathbf{T}_n = \mathrm{tridiag}(1, -2, 1)$.
}\label{bound:err:Fig:3}
\end{figure}

In Figure \ref{bound:err:Fig:3}, in Figure \ref{a-Fig63}, we compare $\Vert \exp(\mathbf{T}_n) - \Psi(\mathbf{T}_n) \Vert_{\infty}$ with the bound from Corollary \ref{approx} for $n = 1:25$, while in other figures, we compare $\Vert \exp(\mathbf{T}_n) - \Psi^{[d]}(\mathbf{T}_n) \Vert_{\infty}$ with the bound in \eqref{bound:error:d:band} for different values of the matrix dimension $n$ and bandwidth $d$. As observed, both bounds are not far from the error arising from the approximation methods.

\subsection{Banded approximation method}
In this example, we examine the numerical test of the banded approximation method presented in Section \ref{sec:band:approx} and the original method introduced in \eqref{approx_matrix} for symmetric case and in Theorem \ref{Hadam} for non-symmetric case. Here, we consider tridiagonal matrices with different dimensions, where their entries are chosen randomly in MATLAB using $ \mathtt{randn(1)} $.

\begin{table}
\caption{Time (seconds) and error in MATLAB for the ``Original methods'' presented in Theorem \ref{Hadam} for non-symmetric matrices, and the method in \eqref{approx_matrix} for symmetric matrices, along with ``Banded Approx.'' presented in \eqref{band:non:sym:approx} for non-symmetric matrices, and the method in \eqref{banded:Psi:new} for the symmetric case, where for both $d = 25$, and the absolute error of the methods is measured in the infinity norm with $\mathtt{expm}$, for tridiagonal matrices with randomly chosen entries.  }
\label{tab:Top:comp}
\centering
\pgfplotstabletypeset[
every head row/.style={
before row={
\toprule
\multicolumn{1}{c||}{} & \multicolumn{1}{c|}{} & \multicolumn{2}{c|}{Original Method} & \multicolumn{2}{c|}{Banded Approx.} & $\mathtt{expm}$ \\
},
after row=\midrule,
},
every row no 6/.style={
after row=\midrule
},
every last row/.style={after row=\bottomrule},
highlightcell/.style={
postproc cell content/.append code={
\ifnum\pgfplotstablerow=#1\relax
\pgfkeysalso{@cell content/.add={$\bf}{$}}
\fi
}
},
columns={Name, 0, 2, 3, 4, 5, 6}, 
columns/Name/.style={column name=Type of matrix, string type, column type=c||}, 
columns/0/.style={column name= Matrix size, column type=c|},
columns/1/.style={column name= $\Delta$, column type=c|},
columns/2/.style={column name=Time, fixed,
},
columns/3/.style={column name=Error, column type=c|, skip rows between index={2}{2}},
columns/4/.style={
column name=Time, fixed
},
columns/5/.style={column name=Error, skip rows between index={2}{2}, column type=c|},
columns/6/.style={
column name=Time,
fixed,
},
]{
Name 0 2 3 4 5 6
{}  500        0.010            2.162e-14            0.009             2.162e-14          0.023
{} 1000        0.025            1.68e-13            0.025             1.68e-13          0.043
{} 3000        0.21            2.742e-13            0.21             2.742e-13          0.702
Symmetric 5000        0.552            1.067e-11            0.548             1.067e-11          2.298
{} 7000        1.207            1.749e-12            1.313             1.749e-12          5.742
{} 9000        2.023            1.959e-12            2.228             1.959e-12         11.509
{} 11000        3.438            1.939e-12            3.436             1.939e-12         21.685
{} 500  0.012            2.33e-13            0.009             2.34e-13          0.066
{} 1000        0.031                  NaN            0.030             8.218e-16          0.134
{} 3000        0.21                  NaN            0.21             6.350e-16          2.160
Non-symmetric  5000        0.877            4.779e-14            0.916             4.779e-14         10.355
{} 7000        1.910            3.31e-16            1.868             3.32e-16         18.094
{} 9000        3.023                  NaN            3.047             5.216e-16         45.63
{} 11000        4.507                  NaN            4.585             6.065e-16         92.962
}
\end{table}

In Table \ref{tab:Top:comp}, we show the results obtained by MATLAB, including the time and error of the original methods, as well as the time and absolute error of the banded approximation with band-width $d = 25$, for symmetric and non-symmetric tridiagonal matrices. The $\mathtt{expm}$ function in MATLAB has been used to calculate the absolute  error of each method with the infinity norm.
As can be seen, for all matrix sizes of symmetric matrices, both methods provide a good approximation with very small error.
However, in the case of the non-symmetric matrix, due to the structure of the matrix $\mathbf{F}_n$ in Theorem \ref{Hadam}, some of its entries become $\mathtt{Inf}$ in MATLAB for certain values of $a$ and $c$ (i.e., values larger than $\mathtt{1e309}$). Similarly, in $\Psi(\mathbf{T}_n)$, some entries become extremely small (i.e., smaller than $\mathtt{1e-324}$), which, when used in the corresponding positions of the Hadamard product, result in $\mathtt{NaN}$. Consequently, the error in the table is displayed as $\mathtt{NaN}$. This problem is solved in the banded approximation method.

\subsection{Crank--Nicolson method}\label{sec-CN}
Here, we compare the Crank--Nicolson method with method presented in $ \eqref{num} $ for two different initial conditions to test efficiency and stability of the method in this paper.

\subsubsection*{Periodic initial condition}
Consider the one-dimensional heat equation on $(0,1)$ with homogeneous boundary conditions, where $a = 1$, and the following initial condition:
\begin{equation}\label{cond:eq:sin}
u(x,0)=\sin(\pi x).
\end{equation}
The exact solution of this heat equation with these conditions is $u(x,t) = e^{-\pi^2 t} \sin(\pi x)$. 

In Figure \ref{fighq}, the absolute errors, using infinity norm, of the presented method in Section \ref{sec:heat:eq1}, method in \eqref{num}, and the Crank--Nicolson method are shown. In this figure, we illustrate the methods for $\Delta x = 0.05$ and $d = 8$, with $\Delta t = 0.04$ (left) and $\Delta t = 0.05$ (right) in time steps. As shown in the figure, the errors in both methods are approximately similar because both methods are of order $2$ in space.

\begin{figure}
\begin{center}
\includegraphics[width=0.49\textwidth]{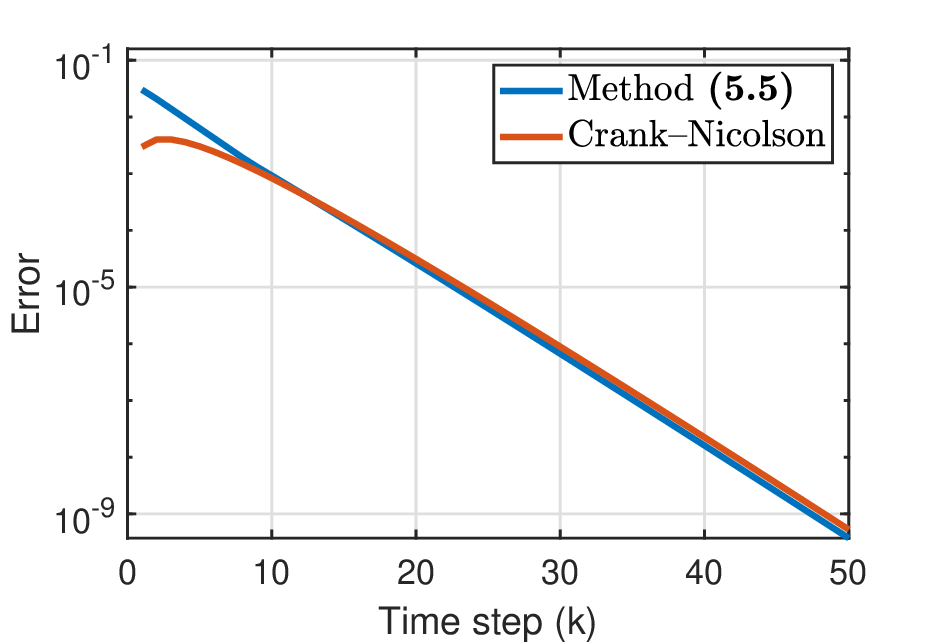}
\includegraphics[width=0.49\textwidth]{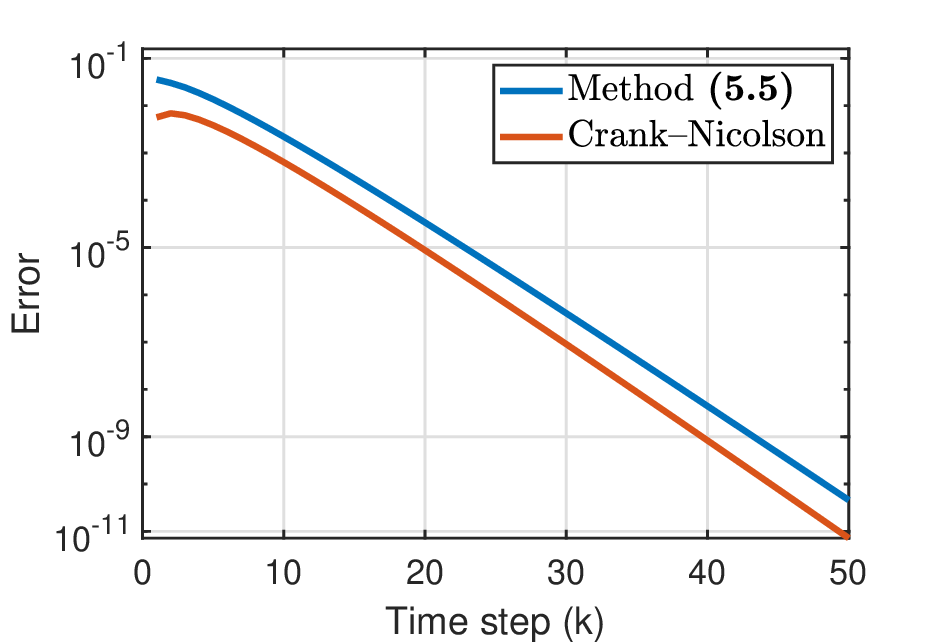}
\caption{The absolute  infinity error of the method in \eqref{num} and the Crank-Nicolson method at $ t_k = k \Delta t $ for $ \Delta t = 0.04 $ (left) and $ \Delta t = 0.05 $ (right), with $ \Delta x = 0.05 $ and $ d = 8 $, where the initial condition is given by $ \eqref{cond:eq:sin} $.}\label{fighq}
\end{center}
\end{figure}

\subsubsection*{Discontinuous initial condition}
 Here, let us consider the following initial condition:
\begin{equation}\label{initail:x:05:cond}
u(x,0) = \begin{cases} 
1,~ & \text{if } x = 0.5, \\
0,~ & \text{otherwise}.
\end{cases}
\end{equation}

In Figure \ref{Maximum}, the results of the Crank--Nicolson method are compared with those of the proposed method in \eqref{num}. From \cite{Morton}, we know that the Crank--Nicolson method suffers from oscillations when $ \mu > 1 $. Here, we set $ \Delta x = 1/21 $, $ \Delta t = 0.005 $, and $ \mu = 2.205 $ with $d =8$. As seen in Figure \ref{Maximum}, the Crank--Nicolson method exhibits oscillatory behavior. As stated in \cite{Morton}, when $ \mu > 1 $, the Crank--Nicolson method dose not preserve maximum principle,  while, according to the stability property in Theorem \ref{thm:stable:max}, the proposed method provides non-oscillatory solutions.

\begin{figure}
\begin{center}
\includegraphics[width=0.49\textwidth]{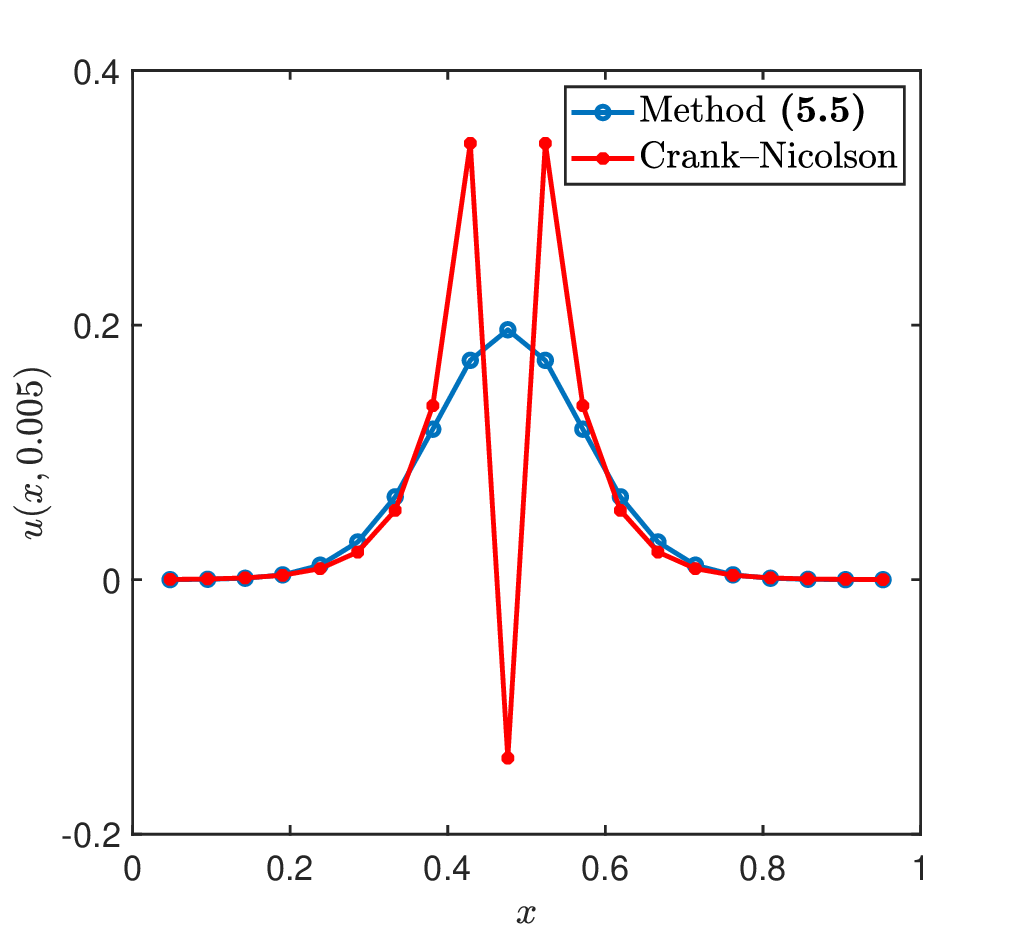}
\includegraphics[width=0.49\textwidth]{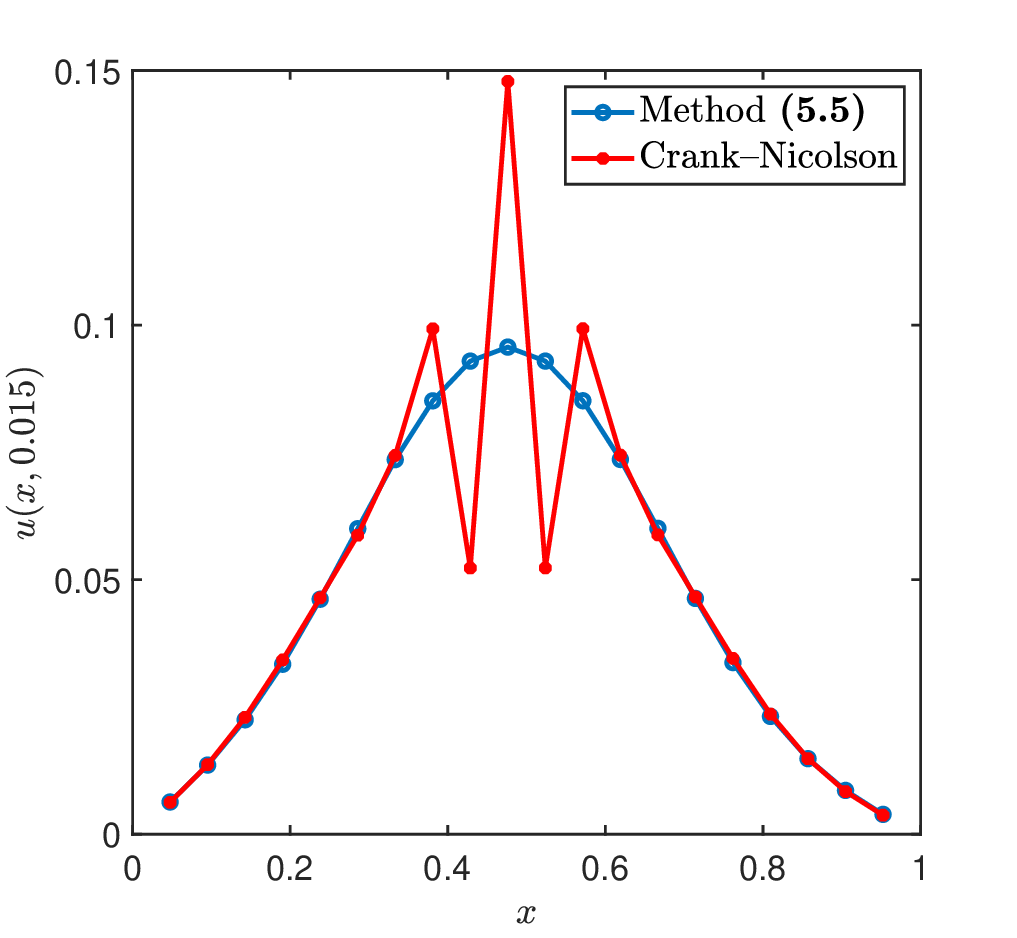}
\caption{The numerical solutions of the one-dimensional heat equation with the initial condition \eqref{initail:x:05:cond}, by using the Crank--Nicolson method and the presented method in \eqref{num} at $ t = 0.005 $ (left) and $ t = 0.015 $ (right), with $ \mu = 2.205 $, $ \Delta x = 1/21 $, and $ d = 8 $.  }\label{Maximum}
\end{center}
\end{figure}

\subsection{Gauss Runnge--Kutta method of order 4}
In this experiment, we aim to compare the Gauss Runnge--Kutta method of order 4  (GRK4) with method presented in \eqref{num}. The Butcher tableau for GRK4 is given by:
\[
\begin{array}{c|cc}
\frac{3-\sqrt{3}}{6} & \frac{1}{4} & \frac{3-2\sqrt{3}}{12} \\[4pt] 
\frac{3+\sqrt{3}}{6} & \frac{3+2\sqrt{3}}{12} & \frac{1}{4} \\[4pt] 
\hline
\\[-8pt]
 & \frac{1}{2} & \frac{1}{2} \\
\end{array}
\]

In this example, we again consider the heat equation with the same conditions assumed in the previous example in Section~\ref{sec-CN}, with initial condition $u(x,0) = \sin(\pi x)$. The exact solution of this problem under these conditions is $u(x,t) = e^{-\pi^2 t} \sin(\pi x)$. By considering the GRK4 method, for solving \eqref{eq:transf:heat} with initial condition $\mathrm{U}_0 = \left(\sin(\pi\Delta x),\ldots,\sin(n\pi\Delta x)\right)^T$, we have:
\[
\mathbf{U}_{G}^{(k+1)} = \mathbf{U}_{G}^{(k)} + \frac{\Delta t}{2} V_1^{(k)} + \frac{\Delta t}{2} V_2^{(k)},
\]
where $V_1^{(k)}$ and $V_2^{(k)}$ are computed as follows:
\begin{equation}\label{vector:V1andV2}
\begin{aligned}
V_1^{(k)} &= \mathbf{K}_n\left(\mathbf{U}_{G}^{(k)}+\frac{\Delta t}{4} V_1^{(k)}+\Delta t\frac{3-2\sqrt{3}}{12} V_2^{(k)}\right),\\
V_2^{(k)} &= \mathbf{K}_n\left(\mathbf{U}_{G}^{(k)}+\Delta t\frac{3+2\sqrt{3}}{12} V_1^{(k)}+\frac{\Delta t}{4} V_2^{(k)}\right).\\
\end{aligned}
\end{equation}
In \eqref{vector:V1andV2}, we need to solve a linear system of size $ (2n) \times (2n) $, where $V_1^{(k)}$ and $V_2^{(k)}$ are unknowns that must be determined at each step. Although here GRK4 provides slightly higher accuracy than method \eqref{num} (see Figure \ref{a-Fig6355}), both methods use a semi-discretization of the second-order spatial derivative and thus have order $\mathcal{O}((\Delta x)^2)$.
The overall complexity of the GRK4 is approximately $ \mathcal{O}(kn^2) $, whereas the complexity of method \eqref{num} is $ \mathcal{O}(2kn\log(n)) $ (see Figure \ref{b-Fig63525}). Here, we set $d = n - 1$ (i.e., a full matrix) in method \eqref{num}.
In Figure \ref{RK:examp:Fig:3}, we illustrate the results for both methods with $ \Delta x = 1/(n+1) $ and $ \Delta t = 0.04 $. In Figure \ref{a-Fig6355}, we show the absolute errors for $ n = 19 $ and time steps $ k = 1:50 $. In Figure \ref{b-Fig63525}, we show the computational time required for both methods at time $ t_{10} = 0.4 $ ($ k = 10 $) and $ n = 1000:4000:25000 $. 
\begin{figure}[t] 
     \centering
     \begin{subfigure}[b]{0.49\textwidth}
         \centering
         \includegraphics[width=\textwidth]{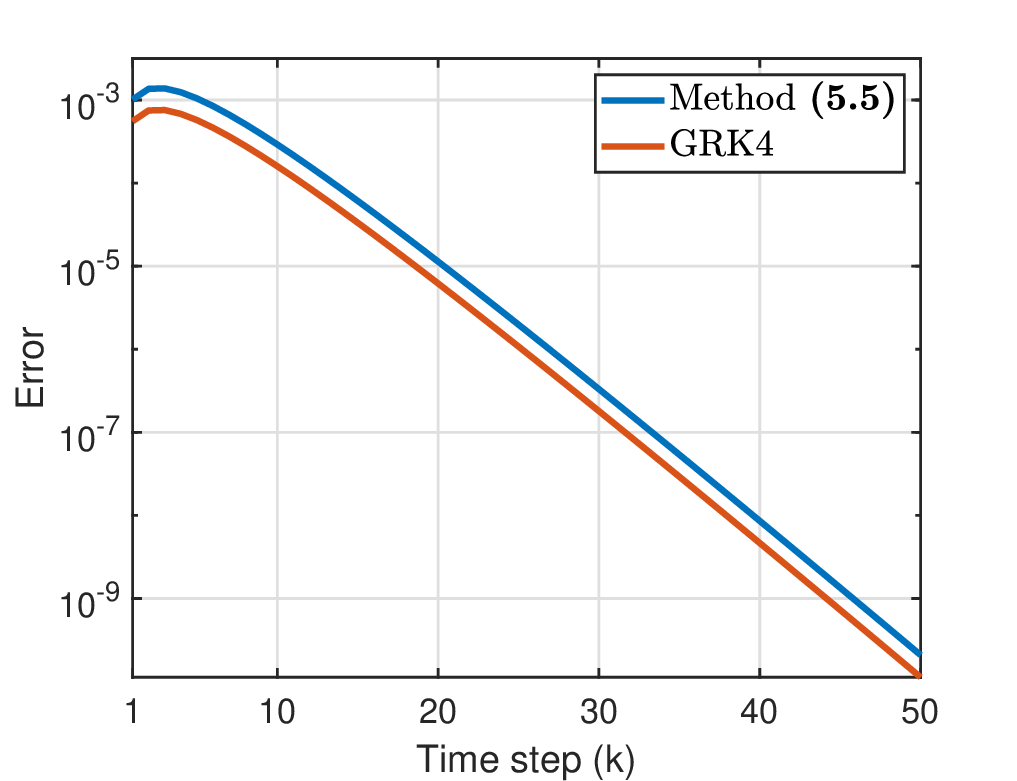}
\caption{$n = 19$ and $k = 1:50$.}\label{a-Fig6355}
     \end{subfigure}
     \hfill
     \begin{subfigure}[b]{0.49\textwidth}
         \centering
         \includegraphics[width=\textwidth]{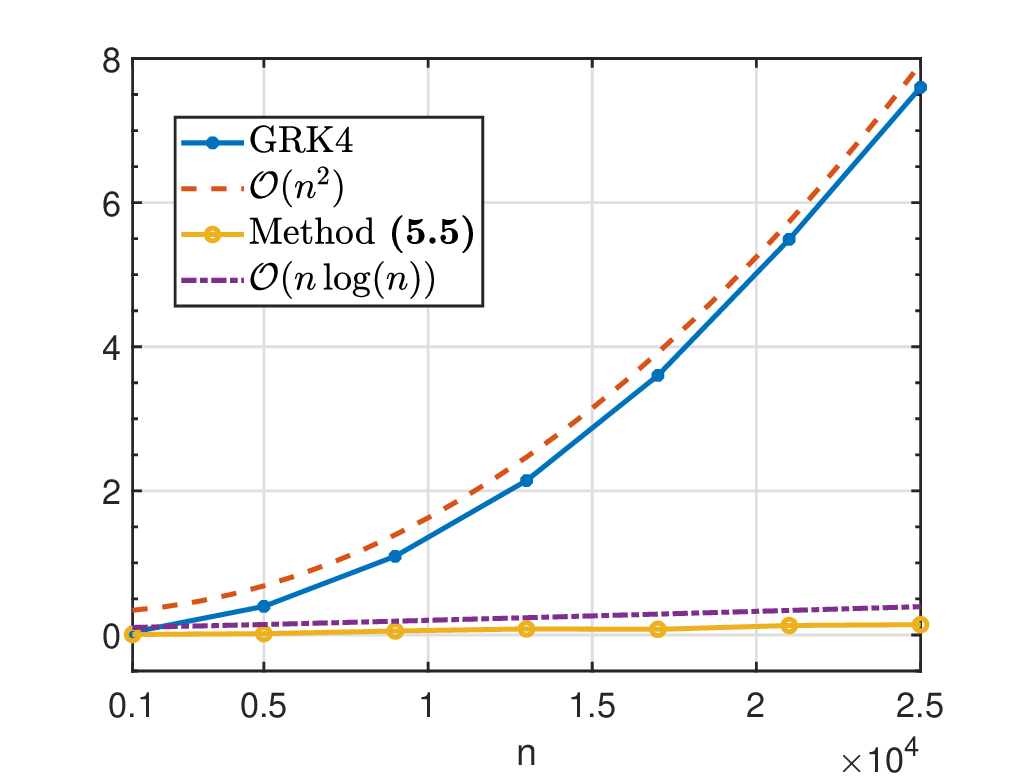}
\caption{ $n = 1000:4000:25000$ and $k = 10$.}\label{b-Fig63525}
     \end{subfigure}
\caption{Comparison of the Gauss Runnge--Kutta method of order 4  (GRK4) with the method presented in $ \eqref{num} $, where in both methods, $ \Delta x = \frac{1}{n+1} $ and $ \Delta t = 0.04 $.
}\label{RK:examp:Fig:3}
\end{figure}

\section{Conclusion}\label{Conclusion}
We analyzed an existing method for approximating the exponential of a symmetric tridiagonal matrix using Bessel functions of the first kind and presented a new and sharp error bound for this approximation. We extended this method to non-symmetric tridiagonal matrices using the Hadamard (element-wise) product and derived an error bound for it. In order to reduce the time complexity and provide a dimension-independent approach, we established bounds for the entries of the matrix exponential of a tridiagonal matrix and used them to approximate the matrix exponential by a banded matrix. By using these results, we proposed a method for solving the heat equation, which unconditionally preserves the maximum principle. Finally, we presented numerical results for the methods introduced here and the existing ones, which demonstrated the efficiency of our approach and the superiority of the obtained bounds. 

\end{document}